\documentclass{elsarticle}
\usepackage{lipsum}
\makeatletter
\def\ps@pprintTitle{%
 \let\@oddhead\@empty
 \let\@evenhead\@empty
 \def\@oddfoot{}%
 \let\@evenfoot\@oddfoot}
\makeatother

\usepackage{graphicx}
\usepackage{hyperref}
\usepackage{amsfonts}
\usepackage{pgf,tikz,pgfplots}
\usepackage{epstopdf}
\usepackage{amssymb,amsmath}
\usepackage{float}
\usepackage[displaymath, mathlines,running]{lineno}
\usepackage{mathtools} 
\usepackage{color}
\usepackage{blkarray}
\usepackage{enumitem}
\usepackage{fmtcount}
\usepackage{multirow}
\usepackage[normalem]{ulem} 
\usepackage{cancel}

\def\dioni#1#2{\textcolor{black}{#1}{#2}}

\newenvironment{proof}{\smallskip\noindent{{\it Proof.}}\hskip \labelsep}%
            {\hfill\penalty10000\raisebox{-.09em}{\large\bf\rm $\blacksquare$}\par\medskip}

\newtheorem{theorem}{Theorem}[section]

\newtheorem{proposition}[theorem]{Proposition}

\newtheorem{definition}{Definition}

\def\xx{\mathbf{x}}
\def\xe{\mathbf{x}^*}
\def\x12e{\mathbf{x}_{\frac12}^*}
\def\cP{\mathcal{P}}
\def\cD{\mathcal{D}}
\def\cR{\mathcal{R}}
\begin{document}
\begin{frontmatter}

\title{A Cell-Average Non-Separable Progressive Multivariate WENO Method for Image Processing Applications}\tnotetext[label1]{The first, second, fifth and sixth authors have been supported by grant PID2023-146836NB-I00 funded by MCIN/AEI/10.13039/501100011033. The fourth author has been supported by NSF grant DMS-2309249.}

\author[UV]{Inmaculada Garc\'es}
\ead{Inmaculada.Garces@uv.es}
\author[UV]{Pep Mulet}
\ead{pep.mulet@uv.es}
\author[UPCT]{Juan Ruiz-\'Alvarez}
\ead{juan.ruiz@upct.es}
\author[UB]{Chi-Wang Shu}
\ead{chi-wang\_shu@brown.edu}
\author[UV]{Dionisio F. Y\'a\~nez}
\ead{dionisio.yanez@uv.es}
\date{Received: date / Accepted: date}

\address[UV]{Departamento de Matem\'aticas, Facultad de Matemáticas, Universidad de Valencia, Valencia, Spain}
\address[UPCT]{Departamento de Matemática Aplicada y Estadística, Universidad  Polit\'ecnica de Cartagena, Cartagena, Spain}
\address[UB]{Division of Applied Mathematics, Brown University, Providence, Rhode Island, USA}
\begin{abstract}

Accurate and efficient reconstruction techniques are essential in multiresolution analysis and image compression, particularly when the data are represented as cell averages. In this work, we present a non-separable progressive multivariate Weighted Essentially Non-Oscillatory (WENO) scheme specifically designed for cell-average data, with applications to digital image processing. The proposed method extends Harten’s multiresolution framework through a non-linear WENO reconstruction adapted to the cell-average context. {\color{black} In contrast to classical WENO schemes, the progressive strategy allows the recursive recovery of high-order accuracy even when the largest stencil is affected by a discontinuity. The method achieves high-order accuracy in smooth regions together with stable}, non-oscillatory behavior near discontinuities. We also establish theoretical results regarding the consistency and approximation properties of the method. Finally, several numerical experiments on piecewise smooth functions and digital images are presented to demonstrate its performance and validate its effectiveness against the linear Lagrange reconstruction of the same order of accuracy.

\end{abstract}

%
%
\begin{keyword}
WENO, non-linear reconstruction, order of accuracy, high accuracy interpolation
\end{keyword}

\end{frontmatter}


\def\S{\mathcal{S}}

\section{Introduction. Harten's multiresolution for cell-average setting in 2D}

In recent years, several nonlinear techniques, such as the weighted essentially non-oscillatory (WENO) algorithm, originally developed for the numerical approximation of hyperbolic partial differential equations (PDEs) \cite{JiangShu}, have been refined by enhancing specific properties of the method. These improvements include better discontinuity detection, monotonicity analysis, and increased order of accuracy. Moreover, WENO methods have been extended to higher-dimensional settings.

In particular, a new progressive WENO scheme was introduced in \cite{MRSY24}, which allows interpolation at arbitrary points, unlike the classical WENO method that only interpolates at midpoints. Additionally, this progressive approach achieves increasingly higher accuracy as the evaluation point moves away from discontinuities. The core idea of the method is to use Aitken-Neville’s procedure (see \cite{gascaaitkenclasica}) to compute optimal weights at each stage (see \cite{ARSY20}), which are then replaced by nonlinear weights through a recursive algorithm.

Several academic tests have been conducted to validate the method’s properties, using examples where the data are interpreted as evaluations of an unknown function at a sequence of nodes. In this paper, we aim to extend these ideas to the context of digital image compression. To this end, we consider the data as cell averages of a function. Specifically, let $f : \Omega = [0,1]^2 \to \mathbb{R}$ be an integrable unknown function, i.e., $f \in L^1(\Omega)$, and let $h_{{\color{black}\ell}} = 1/2^{{\color{black}\ell}}$. Then, our data are

$$\bar{f}^{{\color{black}\ell}}_{i,j}=\frac{1}{|\Omega^{{\color{black}\ell}}_{i,j}|}\int_{\Omega^{{\color{black}\ell}}_{i,j}} f(x,y)dxdy,\qquad i,j=1,\hdots,2^{{\color{black}\ell}},$$
with $\Omega^{{\color{black}\ell}}_{i,j}=[(i-1)h_{{\color{black}\ell}},ih_{{\color{black}\ell}}]\times[(j-1)h_{{\color{black}\ell}},jh_{{\color{black}\ell}}]$ and $|\Omega^{{\color{black}\ell}}_{i,j}|=h_{{\color{black}\ell}}^2$. For this goal, we use Harten's multiresolution (MR) \cite{Har96}, which has been applied in industrial applications in recent years (see, e.g., \cite{cohen}). The MR framework establishes a relationship between two consecutive resolution levels, ${\color{black}\ell}-1$ and ${\color{black}\ell}$, through operators that perform reconstruction and discretization of the available data. Specifically, starting from the data at level ${\color{black}\ell}-1$, we interpolate these values using a reconstruction operator
$$
\mathcal{R}_{{\color{black}\ell}-1}: \mathbb{R}^{2^{{\color{black}\ell}-1} \times 2^{{\color{black}\ell}-1}} \to L^1(\Omega),
$$
which can be, for example, a piecewise polynomial Lagrange interpolator \cite{Har96},
$$
(\mathcal{R}_{{\color{black}\ell}-1} \bar{f}^{{\color{black}\ell}-1})(x,y) = \mathcal{I}(x,y; \bar{f}^{{\color{black}\ell}-1}),
$$
where $\bar{f}^{{\color{black}\ell}-1} = \{\bar{f}^{{\color{black}\ell}-1}_{i,j}\}_{i,j=1}^{2^{{\color{black}\ell}-1}}$ represents the cell average data at level ${\color{black}\ell}-1$. Alternatively, a nonlinear interpolation technique such as the progressive WENO method \cite{MRSY24} adapted to the cell-average context may be used.

The interpolating function $\mathcal{I}$ is constructed by imposing the following conditions:
$$|\Omega_{i,j}^{{\color{black}\ell}-1}|^{-1}\int_{\Omega_{i,j}^{{\color{black}\ell}-1}}\mathcal{I}(x,y;\bar f^{{\color{black}\ell}-1})dxdy=|\Omega_{i,j}^{{\color{black}\ell}-1}|^{-1}\bar f^{{\color{black}\ell}-1}_{i,j}=|\Omega_{i,j}^{{\color{black}\ell}-1}|^{-1}\int_{\Omega_{i,j}^{{\color{black}\ell}-1}}f(x,y)dxdy,\quad i,j=1,\hdots,2^{{\color{black}\ell}-1},$$
Finally, we compute the approximation of the values $\bar{f}^{{\color{black}\ell}}$ by discretizing the reconstructed function $\mathcal{I}$, thus
$$\tilde{f}_{i,j}^{{\color{black}\ell}}=(\mathcal{D}_{{\color{black}\ell}}(\cR_{{\color{black}\ell}-1} \bar{f}^{{\color{black}\ell}-1}))_{i,j}:=|\Omega_{i,j}^{{\color{black}\ell}}|^{-1}\int_{\Omega_{i,j}^{{\color{black}\ell}}}\mathcal{I}(x,y;\bar f^{{\color{black}\ell}-1})dxdy,\quad i,j=1,\hdots,2^{{\color{black}\ell}}.$$
In summary, we have two operators, decimation and prediction, constructed from the composition of the two previously introduced:
$$\cD_{{\color{black}\ell}}^{{\color{black}\ell}-1}\bar{f}^{{\color{black}\ell}}=\cD_{{\color{black}\ell}-1}\cR_{{\color{black}\ell}}\bar{f}^{{\color{black}\ell}},\quad \cP_{{\color{black}\ell}-1}^{{\color{black}\ell}}\bar{f}^{{\color{black}\ell}-1}=\cD_{{\color{black}\ell}}\cR_{{\color{black}\ell}-1}\bar{f}^{{\color{black}\ell}-1},$$
being the discretization operator in this particular case:
\begin{equation*}
\begin{split}
(\cD_{{\color{black}\ell}}^{{\color{black}\ell}-1}\bar{f}^{{\color{black}\ell}})_{i,j}&=h_{{\color{black}\ell}-1}^{-2}\int_{\Omega_{i,j}^{{\color{black}\ell}-1}}\mathcal{I}(x,y,\bar{f}^{{\color{black}\ell}})d(x,y)=\frac{1}{4}\sum_{k_1,k_2=-1}^0h_{{\color{black}\ell}}^{-2}\int_{\Omega_{2i+k_1,2j+k_2}^{{\color{black}\ell}}}\mathcal{I}(x,y,\bar{f}^{{\color{black}\ell}})d(x,y)\\
                                &=\frac14\sum_{k_1,k_2=-1}^0\bar{f}^{{\color{black}\ell}}_{2i+k_1,2j+k_2},
\end{split}
\end{equation*}
expression obtained from the linearity of the integral and
$$\Omega_{i,j}^{{\color{black}\ell}-1}=\bigcup_{k_1,k_2=-1}^0 \Omega_{2i+k_1,2j+k_2}^{{\color{black}\ell}}.$$
Note also that
\begin{equation*}
\begin{split}
\frac14\sum_{k_1,k_2=-1}^0\tilde{f}^{{\color{black}\ell}}_{2i+k_1,2j+k_2}&=\frac{1}{4}\sum_{k_1,k_2=-1}^0h_{{\color{black}\ell}}^{-2}\int_{\Omega_{2i+k_1,2j+k_2}^{{\color{black}\ell}}}\mathcal{I}(x,y,\bar{f}^{{\color{black}\ell}-1})d(x,y)=h_{{\color{black}\ell}-1}^{-2}\int_{\Omega_{i,j}^{{\color{black}\ell}-1}}\mathcal{I}(x,y,\bar{f}^{{\color{black}\ell}-1})d(x,y)\\
                                &=\bar{f}^{{\color{black}\ell}-1}_{i,j},
\end{split}
\end{equation*}
then if we denote the error as $e^{{\color{black}\ell}}=\bar f^{{\color{black}\ell}}-\tilde f^{{\color{black}\ell}}$ we get
$$(\mathcal{D}_{{\color{black}\ell}}^{{\color{black}\ell}-1}e^{{\color{black}\ell}})_{i,j}=\frac14\sum_{k_1,k_2=-1}^0e^{{\color{black}\ell}}_{2i+k_1,2j+k_2}=\frac14\sum_{k_1,k_2=-1}^0\bar{f}^{{\color{black}\ell}}_{2i+k_1,2j+k_2}-\frac14\sum_{k_1,k_2=-1}^0\tilde{f}^{{\color{black}\ell}}_{2i+k_1,2j+k_2}=0.$$
Therefore, if we write $d_1^{{\color{black}\ell}-1}=\{e^{{\color{black}\ell}}_{2i-1,2j-1}\}_{i=1}^{2^{{\color{black}\ell}-1}},\,d_2^{{\color{black}\ell}-1}=\{e^{{\color{black}\ell}}_{2i,2j-1}\}_{i=1}^{2^{{\color{black}\ell}-1}},\,d_3^{{\color{black}\ell}-1}=\{e^{{\color{black}\ell}}_{2i-1,2j}\}_{i=1}^{2^{{\color{black}\ell}-1}}$ then we have the following bijection:
$$\bar f^{{\color{black}\ell}}\equiv \{\bar{f}^{{\color{black}\ell}-1},d_1^{{\color{black}\ell}-1},d_2^{{\color{black}\ell}-1},d_3^{{\color{black}\ell}-1}\}.$$
If we repeat this process several times, we obtain a multiresolution representation of the matrix $\bar f^{{\color{black}\ell}}$.


Since the discretization operator is determined by the nature of the application—namely, image processing—the main objective within the multiresolution framework is to design an appropriate reconstruction operator. Traditionally, piecewise linear interpolation has been used for this purpose, serving as the reconstruction operator $\mathcal{R}_{{\color{black}\ell}}$ (see, for instance, \cite{Har96}). However, this linear approach is prone to producing the Gibbs phenomenon near discontinuities, which can significantly degrade the quality of the reconstruction. To address this issue, various nonlinear strategies have been developed over the past decades. Among them, the Essentially Non-Oscillatory (ENO) method \cite{Har96,Harten1987,MR881365,Abgrall2016xxi} or WENO method \cite{ABM,JiangShu,doi:10.1137/070679065} stand out as a pioneering techniques for avoiding spurious oscillations near discontinuities.

Building on the ENO method, the WENO technique was developed to enhance both accuracy and robustness by combining multiple interpolation stencils with adaptively computed weights \cite{ABM,JiangShu,doi:10.1137/070679065}. These nonlinear approaches have proven highly effective in preserving sharp features near discontinuities while maintaining high-order accuracy in smooth regions. {\color{black} Unlike classical WENO schemes, whose effective order may deteriorate when the largest stencil is contaminated by a discontinuity, the progressive strategy recursively enlarges the stencil and allows high-order accuracy to be recovered in such situations.}

In this work, we introduce a methodology that enables the adaptation of interpolation schemes, originally developed in the framework of pointwise data, to the setting of cell averages. Building upon this formulation, we incorporate the aforementioned WENO strategy with the purpose of applying the multiresolution framework to the problem of digital image compression.

This paper is organized as follows: Section \ref{pvtoca} is devoted to relating the point-value discretization with the cell-average in order to propose a new progressive method for this type of data. In addition, we prove some properties in the cell-average context if the constructed operators satisfy some conditions. In Section \ref{sec:multivariatelinearandwenoclassic}, we review the linear and classical WENO method in the Catersian grid. In the following section, we introduce the new progressive method in two dimensions (similar in several dimensions) and we explain the recursive algorithm to obtain the interpolator. We explicitly present in Section \ref{r3} an example for order $r=3$ to illustrate the simplicity of the algorithm and the method; for any $r$, it would be entirely similar. Some numerical experiments for the cell-averages of functions and for the compression of digital images are presented in Section \ref{expnum}. Finally, some conclusions and future work are shown in the last section.

\section{From point-value interpolation to cell-average interpolation}\label{pvtoca}

In this section, we present an approach to perform interpolation on cell averages based on interpolation from data given in the point-value discretization. The construction is carried out under the assumption that the nodes are equally spaced and located within the square $[0,1]^2$. The extension to any other square in $\mathbb{R}^2$ follows in a similar manner. Therefore, we suppose $f:[0,1]^2\to \mathbb{R}$ an unknown function with $f\in L^1([0,1]^2)$, we consider that our data are
$$\bar{f}_{i,j}^{{\color{black}\ell}-1}=h^{-2}_{{\color{black}\ell}-1}\int_{\Omega_{i,j}^{{\color{black}\ell}-1}}f(x,y)dxdy, \quad i,j=1,\hdots,2^{{\color{black}\ell}-1}=\frac{1}{h_{{\color{black}\ell}-1}},$$
with $\Omega_{i,j}^{{\color{black}\ell}-1}=[x^{{\color{black}\ell}-1}_{i-1},x^{{\color{black}\ell}-1}_i]\times[y^{{\color{black}\ell}-1}_{j-1},y^{{\color{black}\ell}-1}_j]$, $x^{{\color{black}\ell}-1}_i=ih_{{\color{black}\ell}-1}$, $y^{{\color{black}\ell}-1}_j=jh_{{\color{black}\ell}-1}$, and we want to calculate an approximation to
$$\bar{f}^{{\color{black}\ell}}_{2i+k_1,2j+k_2}=h^{-2}_{{\color{black}\ell}}\int_{\Omega_{2i+k_1,2j+k_2}^{{\color{black}\ell}}}f(x,y)dxdy, \quad k_1,k_2=-1,0.$$
Now we focus our attention on the case $k_1=0$, $k_2=0$, as the rest are similar.

It is clear that if we define the primitive function (see \cite{Har96}), $F:[0,1]^2\to \mathbb{R}$ as
\begin{equation}\label{funprimitiva}
F(x,y)=\int_0^x\int_0^y f(s,t)dsdt,
\end{equation}
we have that
\begin{equation}\label{fijcell}
\bar{f}_{i,j}^{{\color{black}\ell}-1}=\frac{1}{h^2_{{\color{black}\ell}-1}}\big(F(x^{{\color{black}\ell}-1}_{i},y^{{\color{black}\ell}-1}_{j})-F(x^{{\color{black}\ell}-1}_{i-1},y^{{\color{black}\ell}-1}_{j})-F(x^{{\color{black}\ell}-1}_{i},y^{{\color{black}\ell}-1}_{j-1})+F(x^{{\color{black}\ell}-1}_{i-1},y^{{\color{black}\ell}-1}_{j-1})\big),
\end{equation}
and also
\begin{equation}\label{fijcell2}
F^{{\color{black}\ell}-1}_{i,j}:=F(x^{{\color{black}\ell}-1}_i,y^{{\color{black}\ell}-1}_j)=h^2_{{\color{black}\ell}-1}\sum_{s,t=1}^{i,j}\bar{f}^{{\color{black}\ell}-1}_{s,t}.
\end{equation}
Thus, we have a relation between our data and the point-value discretization of the primitive function $F$. Now, we use a linear or nonlinear method to calculate an interpolator $\mathcal{I}$ which satisfies the conditions:
\begin{equation}\label{interpolador}
\mathcal{I}(x^{{\color{black}\ell}-1}_i,y^{{\color{black}\ell}-1}_j;F^{{\color{black}\ell}-1})=F^{{\color{black}\ell}-1}_{i,j}, \quad i,j=1,\hdots,2^{{\color{black}\ell}-1},
\end{equation}
and $\mathcal{I}\in \mathcal{C}^2([0,1]^2)$ (we could relax this condition) and define the reconstruction operator as:
\begin{equation}\label{opreconscell}
(\cR_{{\color{black}\ell}-1} \bar{f}^{{\color{black}\ell}-1})(x,y)=\frac{\partial^2}{\partial x\partial y}\mathcal{I}(x,y;F^{{\color{black}\ell}-1}).
\end{equation}
We  suppose that if $(x^*,y^*)\in [0,1]^2$ there exists $s>2$ such that
$$\mathcal{I}(x^*,y^*;F^{{\color{black}\ell}-1})-F(x^*,y^*)=O(h_{{\color{black}\ell}-1}^s).$$
With these ingredients, from Eq. \eqref{fijcell}, $x_{2i}^{{\color{black}\ell}}=x_{i}^{{\color{black}\ell}-1}$ and $y_{2j}^{{\color{black}\ell}}=y_{j}^{{\color{black}\ell}-1}$, we get
\begin{equation*}
\begin{split}
\bar{f}^{{\color{black}\ell}}_{2i,2j}-h_{{\color{black}\ell}}^{-2}\int^{x^{{\color{black}\ell}}_{2i}}_{x^{{\color{black}\ell}}_{2i-1}}\int^{y^{{\color{black}\ell}}_{2j}}_{y^{{\color{black}\ell}}_{2j-1}}\frac{\partial}{\partial x\partial y}\mathcal{I}(x,y;F^{{\color{black}\ell}-1})dxdy=&
h_{{\color{black}\ell}}^{-2}\big(F(x^{{\color{black}\ell}}_{2i},y^{{\color{black}\ell}}_{2j})-F(x^{{\color{black}\ell}}_{2i-1},y^{{\color{black}\ell}}_{2j})-F(x^{{\color{black}\ell}}_{2i},y^{{\color{black}\ell}}_{2j-1})+F(x^{{\color{black}\ell}}_{2i-1},y^{{\color{black}\ell}}_{2j-1})\big)\\
&-h_{{\color{black}\ell}}^{-2}\big(\mathcal{I}(x^{{\color{black}\ell}}_{2i},y^{{\color{black}\ell}}_{2j};F^{{\color{black}\ell}-1})-\mathcal{I}(x^{{\color{black}\ell}}_{2i-1},y^{{\color{black}\ell}}_{2j};F^{{\color{black}\ell}-1})\\
&\,\,\quad\qquad-\mathcal{I}(x^{{\color{black}\ell}}_{2i},y^{{\color{black}\ell}}_{2j-1};F^{{\color{black}\ell}-1})+\mathcal{I}(x^{{\color{black}\ell}}_{2i-1},y^{{\color{black}\ell}}_{2j-1};F^{{\color{black}\ell}-1})\big)\\
=&O(h_{{\color{black}\ell}-1}^{s-2}).
\end{split}
\end{equation*}
It is clear that, due to the dimensionality, if we use this technique without adding any further conditions, the order of approximation decreases to $(s-2)$. The most evident example arises when applying polynomial interpolation based on $(2s-1)\times(2s-1)$ cells, that is, $2s\times 2s$ points if we use the primitive function $F$, Eq. \eqref{funprimitiva}. In that case, we obtain $O(h^{2s-2})$, whereas it is well known that the expected order should be $O(h^{2s-1})$.  For this reason, we introduce the following definition to construct reconstruction operators with optimal order of accuracy. 
\begin{definition}\label{cellconsistent}
An operator
$$\mathcal{R}_{{\color{black}\ell}-1}:\mathbb{R}^{2^{{\color{black}\ell}-1}\times 2^{{\color{black}\ell}-1}}\to L^1([0,1]^2)$$
defined as
$$\mathcal{R}_{{\color{black}\ell}-1}(\bar f^{{\color{black}\ell}-1})(x,y)=\sum_{(i_1,i_2)\in I} a_{i_1,i_2}(x,y)\bar f^{{\color{black}\ell}-1}_{i_1,i_2},$$
being $I\subseteq \{1,\hdots,2^{{\color{black}\ell}-1}\}^2$, $\bar f^{{\color{black}\ell}-1}=\{\bar f^{{\color{black}\ell}-1}_{i,j}\}_{i,j=1}^{2^{{\color{black}\ell}-1}}$ and $a_{i_1,i_2}\in L^1([0,1]^2)$ for all $(i_1,i_2)\in I$, is cell-consistent if there exists a constant $K>0$ such that
\begin{enumerate}
\item $\int_{{\Upsilon}}\sum_{(i_1,i_2)\in I}|a_{i_1,i_2}(x,y)|\leq K$ for all compact $\Upsilon \subseteq [0,1]^2$.
\item $h_{{\color{black}\ell}-1}^{-2}\int_{\Omega^{{\color{black}\ell}-1}_{i,j}}\mathcal{R}_{{\color{black}\ell}-1}(\bar f^{{\color{black}\ell}-1})(x,y)dxdy=\bar{f}_{i,j}^{{\color{black}\ell}-1},\quad i,j=1,\hdots,2^{{\color{black}\ell}-1}.$
\item $h_{{\color{black}\ell}}^{-2}\int_{\Omega^{{\color{black}\ell}}_{2i+k_1,2j+k_2}}\mathcal{R}_{{\color{black}\ell}-1}(\bar p^{{\color{black}\ell}-1})(x,y)dxdy=\bar p^{{\color{black}\ell}}_{2i+k_1,2j+k_2},\quad \forall p\in \Pi_N$.
\end{enumerate}
%
%
With $k_1,k_2=-1,0$, $N\in\mathbb{N}$ and {\color{black} $\Pi_N$ stands for bivariate polynomials of total degree less than or equal to $N$.}
\end{definition}
Thus, we can prove the following proposition {\color{black} inspired from the theory} proposed in \cite{levin}.
\begin{proposition}
{\color{black} Let $f\in\mathcal{C}^{N+1}([0,1]^2)$. If $\mathcal{R}_{\ell-1}$ is cell-consistent, then}
$$\left|\bar f^{{\color{black}\ell}}_{2i+k_1,2j+k_2}-h_{{\color{black}\ell}}^{-2}\int_{\Omega^{{\color{black}\ell}}_{2i+k_1,2j+k_2}}\mathcal{R}_{{\color{black}\ell}-1}(\bar f^{{\color{black}\ell}-1})(x,y)dxdy\right|=O(h^{N+1}_{{\color{black}\ell}-1}),$$
with $k_1,k_2=-1,0$.
\end{proposition}

\begin{proof}

From $f\in\mathcal{C}^{N+1}(\Omega)$ we know that there exist $(\xi_1,\xi_2)\in [0,1]^2$ such that
$$f(x,y)=T(x,y)+\sum_{s+t = N+1} \frac{1}{s!\,t!}
    \frac{\partial^{\,s+t} f}{\partial x^s \partial y^t}(\xi_1,\xi_2)\,x^sy^t,$$
with $T\in \Pi_N$. Note that, if $(i_1,i_2)\in \{1,\hdots,2^{{\color{black}\ell}-1}\}^2$
$$\left|\bar{T}^{{\color{black}\ell}-1}_{i_1,i_2}-\bar{f}^{{\color{black}\ell}-1}_{i_1,i_2}\right|=\left| h_{{\color{black}\ell}-1}^{-2}\int_{\Omega_{i_1,i_2}^{{\color{black}\ell}-1}}(T(x,y)-f(x,y))dxdy\right|=O(h_{{\color{black}\ell}-1}^{N+1}).$$
Without loss of generality $k_1=k_2=0$, similar for the rest of the cases, then
\begin{equation*}
\begin{split}
&\left|\bar f^{{\color{black}\ell}}_{2i,2j}-h_{{\color{black}\ell}}^{-2}\int_{\Omega^{{\color{black}\ell}}_{2i,2j}}\mathcal{R}_{{\color{black}\ell}-1}(\bar f^{{\color{black}\ell}-1})(x,y)dxdy\right|=\left|h_{{\color{black}\ell}}^{-2}\int_{\Omega^{{\color{black}\ell}}_{2i,2j}}\left(f(x,y)-\sum_{(i_1,i_2)\in I} a_{i_1,i_2}(x,y)\bar f^{{\color{black}\ell}-1}_{i_1,i_2}\right)dxdy\right|\\
&=\left|h_{{\color{black}\ell}}^{-2}\int_{\Omega_{2i,2j}}\left(T(x,y)+\sum_{s+t = N+1} \frac{1}{s!\,t!}
    \frac{\partial^{\,s+t} f}{\partial x^s \partial y^t}(\xi_1,\xi_2)\,x^sy^t-\sum_{(i_1,i_2)\in I} a_{i_1,i_2}(x,y)\bar f^{{\color{black}\ell}-1}_{i_1,i_2}\right)dxdy\right|\\
&=\left|h_{{\color{black}\ell}}^{-2}Rh_{{\color{black}\ell}}^{N+3}+h_{{\color{black}\ell}}^{-2}\int_{\Omega^{{\color{black}\ell}}_{2i,2j}}\left(\sum_{(i_1,i_2)\in I} a_{i_1,i_2}(x,y)\bar{T}^{{\color{black}\ell}-1}_{i_1,i_2}-\sum_{(i_1,i_2)\in I} a_{i_1,i_2}(x,y)\bar f^{{\color{black}\ell}-1}_{i_1,i_2}\right)dxdy\right|\\
&\leq h_{{\color{black}\ell}}^{-2}|R|h_{{\color{black}\ell}}^{N+3}+\left|h_{{\color{black}\ell}}^{-2}\int_{\Omega^{{\color{black}\ell}}_{2i,2j}}\left(\sum_{(i_1,i_2)\in I} a_{i_1,i_2}(x,y)(\bar{T}^{{\color{black}\ell}-1}_{i_1,i_2}-\bar f^{{\color{black}\ell}-1}_{i_1,i_2})\right)dxdy\right|\\
&\leq h_{{\color{black}\ell}}^{-2}|R|h_{{\color{black}\ell}}^{N+3}+h_{{\color{black}\ell}}^{-2}\int_{\Omega^{{\color{black}\ell}}_{2i,2j}}\left(\sum_{(i_1,i_2)\in I}\left| a_{i_1,i_2}(x,y)\right|\left|\bar{T}^{{\color{black}\ell}-1}_{i_1,i_2}-\bar f^{{\color{black}\ell}-1}_{i_1,i_2}\right|\right)dxdy\\
&\leq h_{{\color{black}\ell}}^{-2}|R|h_{{\color{black}\ell}}^{N+3}+Kh_{{\color{black}\ell}-1}^{N+1}\\
&=(2^{-(N+1)}|R|+K)h_{{\color{black}\ell}-1}^{N+1}.
\end{split}
\end{equation*}
{\color{black} Here, $R$ denotes the remainder term of the Taylor expansion, and $K$ is the non-negative constant from Definition \ref{cellconsistent}.}
\end{proof}

\begin{proposition}
Let $\mathcal{I}(x,y;F^{{\color{black}\ell}-1})$ be an interpolator which satisfies the conditions imposed in Eq. \eqref{interpolador} defined as
$$ \mathcal{I}(x,y;F^{{\color{black}\ell}-1})=\sum_{(i_1,i_2)\in I} \tilde a_{i_1,i_2}(x,y)\bar f^{{\color{black}\ell}-1}_{i_1,i_2},$$
(note that $F^{{\color{black}\ell}-1}\equiv \bar f^{{\color{black}\ell}-1}$) with $\tilde a_{i_1,i_2}\in \mathcal{C}^2([0,1]^2)$, and
$$\mathcal{I}(x,y;P^{{\color{black}\ell}-1})=P(x,y),\quad \forall\, P\in \Pi_{N+1},$$
then the operator
$$(\cR_{{\color{black}\ell}-1} \bar{f}^{{\color{black}\ell}-1})(x,y)=\frac{\partial^2}{\partial x\partial y}\mathcal{I}(x,y;F^{{\color{black}\ell}-1}),$$
is cell consistent.
\end{proposition}
\begin{proof}
The first and second conditions of the Def. \eqref{cellconsistent} are direct. We prove the third condition with $k_1=k_2=0$ (similar for the rest of the cases), thus let $p\in \Pi_N$ and
$$P(x,y)=\int_0^x\int_0^yp(s,t)dsdt$$
then $P\in\Pi_{N+1}$ and it satisfies
\begin{equation*}
\begin{split}
h_{{\color{black}\ell}}^{-2}\int\limits_{\Omega^{{\color{black}\ell}}_{2i,2j}}\mathcal{R}_{{\color{black}\ell}-1}(\bar p^{{\color{black}\ell}-1})(x,y)dxdy&=h_{{\color{black}\ell}}^{-2}\int_{\Omega^{{\color{black}\ell}}_{2i,2j}}\frac{\partial^2}{\partial x\partial y}\mathcal{I}(x,y;P^{{\color{black}\ell}-1})dxdy\\
&=h^{-2}_{{\color{black}\ell}}(\mathcal{I}(x_{2i},y_{2j};P^{{\color{black}\ell}-1})-\mathcal{I}(x_{2i-1},y_{2j};P^{{\color{black}\ell}-1})-\mathcal{I}(x_{2i},y_{2j-1};P^{{\color{black}\ell}-1})+\mathcal{I}(x_{2i-1},y_{2j-1};P^{{\color{black}\ell}-1})) \\
&=h^{-2}_{{\color{black}\ell}}(P(x_{2i},y_{2j})-P(x_{2i-1},y_{2j})-P(x_{2i},y_{2j-1})+P(x_{2i-1},y_{2j-1}))\\
&=h^{-2}_{{\color{black}\ell}}\int_{x_{2i-1}}^{x_{2i}}\int_{y_{2j-1}}^{y_{2j}}p(s,t)dsdt\\
&=\bar p^{{\color{black}\ell}}_{2i,2j}.
\end{split}
\end{equation*}

\end{proof}

The relation between the reconstruction operator for point-value and cell-average data given in Eq. \eqref{opreconscell} allows us to use the existing classical methods and the new one published in \cite{MRSY24}. Also, we could use other techniques as, for example, radial basis functions (RBF), partition of unity method or nonlinear partition of unity method (PUM) (see e.g. \cite{cavo16,FASSHAUER, wendland}) but we focus our objective on applying the novel progressive WENO-$2r$ scheme in the context of digital image compression. For this purpose, the following section is devoted to reviewing the scheme, concentrating on two variables, although its extension  to higher-dimensional settings can be achieved without significant {\color{black} difficulties}.

\section{Bivariate linear Lagrange interpolation and non-linear WENO method in Cartesian grids}\label{sec:multivariatelinearandwenoclassic}

This section provides a concise overview of the bivariate Lagrange interpolation problem in the context of data distributed on Cartesian grids, and develops the bivariate WENO scheme introduced in \cite{arandigamuletrenau}. In particular, we focus on the interpolation in $[0,1]^2$. The essential components required for extending the new WENO-$2r$
algorithm, presented in Section \ref{weno2d}, to the multivariate setting are also presented.

\subsection{Bivariate linear interpolation}\label{linealnd}
We suppose an equally spaced grid in $[0,1]^2$ denoted by $\{(ih_{{\color{black}\ell}},jh_{{\color{black}\ell}})\}_{i,j=0}^{2^{{\color{black}\ell}}}$ with $h_{{\color{black}\ell}}=2^{-{\color{black}\ell}}$, and an unknown function
 $f:[0,1]^2 \to \mathbb{R}$, and consider our data as the evaluation in the points of the Cartesian grid of this function
$$f^{{\color{black}\ell}}_{i,j}=f(ih_{{\color{black}\ell}},jh_{{\color{black}\ell}}), \quad 0\leq i,j\leq 2^{{\color{black}\ell}}.$$
Let $r\in\mathbb{N}$ and let $(i_0,j_0)\in\{r,2^{{\color{black}\ell}}-r\}^2$ denote the reference index about which the approximation is centered, and define the centered stencil:
\begin{equation}\label{eq1stencil}
\begin{split}
\mathcal{S}_{\mathbf{0}}^{2r}&=\{(i_0-r)h_{{\color{black}\ell}},\hdots,(i_0+r-1)h_{{\color{black}\ell}}\}\times\{(j_0-r)h_{{\color{black}\ell}},\hdots,(j_0+r-1)h_{{\color{black}\ell}}\}.
\end{split}
\end{equation} 
{\color{black} Then, the problem consists in determining a bivariate polynomial of degree $2r-1$ in each variable (i.e., of total degree $4r-2$), constructed as a tensor product of univariate Lagrange polynomials, such that:}
$$p(\xx)=f(\xx), \quad \, \forall\, \xx \in \mathcal{S}_{\mathbf{0}}^{2r},$$
whose unique solution is
$$p_{\mathbf{0}}^{2r-1}(x,y)=\sum_{i=i_0-r}^{i_0+r-1}\sum_{j=j_0-r}^{j_0+r-1}{\color{black}f(ih,jh)}L_{i}(x) L_{j}(y),$$
being the Lagrange base of polynomials:
$$L_{\kappa}(x)=\prod_{k=\kappa_0-r,k\neq \kappa}^{\kappa_0+r-1}\left(\frac{x-kh_{{\color{black}\ell}}}{\kappa h_{{\color{black}\ell}}-kh_{{\color{black}\ell}}}\right), \quad \kappa=\kappa_0-r,\hdots,\kappa_0+r-1,\quad \kappa=i,j. $$
In order to calculate the error, we use the result proved in \cite{arandigamuletrenau}: If $\xe=(x^*,y^*)\in [0,1]^2$, and $f\in \mathcal{C}^{4r}(\mathbb{R}^2)$ then the interpolation error is:
\begin{equation}\label{errorfuncion}
E(\xe)=f(\xe)-p_{\mathbf{0}}^{2r-1}(\xe)= O(h_{{\color{black}\ell}}^{2r}).
\end{equation}

\subsection{Bivariate WENO interpolation in Cartesian grids}\label{wenoclasicond}

In this subsection, we revisit the WENO interpolant introduced in \cite{arandigamuletrenau}; the key idea is to employ the same ingredients that are typically used in the construction of the WENO scheme in one dimension.

We consider the point where we want to approximate the function
$$\xe=(x^*,y^*)\in\{(i_0-1/2,j_0-1/2)h_{{\color{black}\ell}},(i_0,j_0-1/2)h_{{\color{black}\ell}},(i_0-1/2,j_0)h_{{\color{black}\ell}}\} \subset [0,1]^2,$$
we design a convex combination of interpolants of lower degree, $r$, {\color{black} let $\mathbf{k}=(k_1,k_2)$ and define}
$$p_0^{2r-1}(\xe)=\sum_{{\color{black}\mathbf{k}}\in \{0,1,\hdots,r-1\}^2} C^r_{{\color{black} \mathbf{k}}}p_{{\color{black}\mathbf{k}}}^r(\xe),$$
where $C^r_{{\color{black}\mathbf{k}}}=C^r_{k_1}C^r_{k_2}$ are the optimal weights whose explicit formula is showed in \cite{ABM}:
\begin{equation}\label{opt_w}
{C}_{k}^r=\frac{1}{2^{2r-1}}
\binom{2r}{2k+1}
,  \quad k=0,\cdots, r-1.
\end{equation}
The polynomials $p_{{\color{black} \mathbf{k}}}^r$ interpolate at the nodes
\begin{equation*}
\begin{split}
\mathcal{S}_{{\color{black}\mathbf{k}}}^{r}&=\{(i_0+k_1-r)h_{{\color{black}\ell}},\hdots,(i_0+k_1+r)h_{{\color{black}\ell}}\}\times\{(j_0+k_2-r)h_{{\color{black}\ell}},\hdots,(j_0+k_2+r)h_{{\color{black}\ell}}\}=\mathcal{S}_{k_1}^r \times \mathcal{S}_{k_2}^r.
\end{split}
\end{equation*}
The optimal weights are subsequently replaced by non-linear ones, as defined by the following expression \cite{MRSY24}
\begin{equation}\label{omegand}
\omega^r_{\mathbf{k}}(\xe)=\frac{\alpha^r_{\mathbf{k}}(\xe)}{\sum_{\mathbf{l}\in\{0,\hdots,r-1\}^2}\alpha^r_{\mathbf{l}}(\xe)},\,\,\text{with}\,\,
\alpha^r_{\mathbf{k}}(\xe)=\frac{C^r_{\mathbf{k}}(\xe)}{(\epsilon+I^r_{\mathbf{k}})^t},
\end{equation}
where $I^r_{\mathbf{k}}$ are the smoothness indicators, subject to conditions:
\begin{enumerate}[label={\bfseries P\arabic*}]
\item\label{P1sm} The order of the smoothness indicator in regions free of discontinuities is $h^2$, i.e.,
$$I^r_{\mathbf{k}}=O(h_{{\color{black}\ell}}^2) \,\, \text{if}\,\, f \,\, \text{is smooth in } \,\, \mathcal{S}^r_{\mathbf{k}}.$$
\item\label{P2sm}The distance between two smoothness indicators in regions free of discontinuities is $h_{{\color{black}\ell}}^{r+1}$, i.e.,
 let
  $\mathbf{k}=(k_1,k_2)$, and $\mathbf{k}'=(k'_1,k'_2)$
 be such that both $\mathcal{S}^r_{\mathbf{k}}$ and $\mathcal{S}^r_{\mathbf{k}'}$ are free of discontinuities, then
$$I^r_{\mathbf{k}}-I^r_{\mathbf{k}'}=O(h_{{\color{black}\ell}}^{r+1}).$$
\item\label{P3sm} When the stencil $\mathcal{S}^r_{\mathbf{k}}$ is affected by a discontinuity,
 then:
$$I^r_{\mathbf{k}} \nrightarrow 0 \,\, \text{as}\,\, h_{{\color{black}\ell}}\to 0 \,\,(\equiv {{\color{black}\ell}}\to\infty).$$
\end{enumerate}



The smoothness indicators to be designed are those adapted from \cite{MRSY24} to the context of cell-average data available in our setting.
The parameters $\epsilon$ and $t$ are chosen to avoid divisions by zero ant to get the maximum order at smooth zones respectively. In \cite{ABM}, the next theorem is proved with fixed parameters.

\begin{theorem}
Let $\xe\in \{(i^{{\color{black}\ell}}_0-1/2,j^{{\color{black}\ell}}_0-1/2)h_{{\color{black}\ell}},(i^{{\color{black}\ell}}_0,j^{{\color{black}\ell}}_0-1/2)h_{{\color{black}\ell}},(i^{{\color{black}\ell}}_0-1/2,j^{{\color{black}\ell}}_0)h_{{\color{black}\ell}}\}$, the bivariate WENO interpolant
\begin{equation}
\mathcal{I}^{2r-1}(\xe;f)=\sum_{\mathbf{k}\in \{0,1,\hdots,r-1\}^2}\omega^r_{\mathbf{k}}p_{\mathbf{k}}^r(\xe)
\end{equation}
with $\omega^r_{\mathbf{k}}$, $\mathbf{k}\in \{0,1,\hdots,r-1\}^2$
defined in Eq. \eqref{omegand}, with smoothness indicator
$I^r_{\mathbf{k}}$, $\mathbf{k}\in \{0,1,\hdots,r-1\}^2$ \dioni{}{fulfilling}
\ref{P1sm}, \ref{P2sm}, \ref{P3sm}; $\epsilon=h^2$ and
$t=\frac{1}{2}(r+1)$ satisfies:
$$f(\xe)-\mathcal{I}^{2r-1}(\xe;f)=
\begin{cases}
O(h^{2r}), & \text{at smooth regions,}\\
O(h^{r+1}), & \text{if, at least, one stencil lies in a smooth region.}
\end{cases}
$$
\end{theorem}

In the following section, we review the progressive WENO method presented in \cite{MRSY24}, focusing on the case of two variables and performing interpolation at the midpoints of the intervals.

\section{The new progressive bivariate WENO method}\label{weno2d}

The purpose of this approach is to attain the highest possible order of accuracy in situations where a discontinuity intersects the largest stencil. To this end, we employ the two-dimensional version of the Aitken–Neville algorithm \cite{gascaaitkenclasica}. The procedure begins with a polynomial of degree $2r-1$, which is decomposed into $4$ polynomials of degree $2r-2$. This yields linear polynomials as weights, later replaced by non-linear ones that depend on the position of the discontinuity. The process is then repeated: each of the $4$ polynomials of degree $2r-2$ is further decomposed into polynomials of degree $2r-3$, and so forth. At every stage, the non-linear weights identify the stencils unaffected by the discontinuity, which are subsequently employed for the approximation. For clarity, we restrict the presentation in this section to the case $r = 3$, since the extension to general $r$ follows analogously.

Therefore, we consider a uniform  grid in $[0,1]^2$
defined by $\{(x^{{\color{black}\ell}}_i,y^{{\color{black}\ell}}_j)=(ih_{{\color{black}\ell}},jh_{{\color{black}\ell}})\}_{i,j=0}^{2^{{\color{black}\ell}}}$, with $h_{{\color{black}\ell}}=2^{-{\color{black}\ell}}$ and $(i_0,j_0)\in \mathbb{N}^2$ such that $0\leq i_0-2,j_0-2$ and $i_0+1,j_0+1\leq 2^{{\color{black}\ell}}$. We establish the largest stencil
$$\mathcal{S}^{6}_{(0,0)}=\{x^{{\color{black}\ell}}_{i_0-3},x^{{\color{black}\ell}}_{i_0-2},x^{{\color{black}\ell}}_{i_0-1},x^{{\color{black}\ell}}_{i_0},x^{{\color{black}\ell}}_{i_0+1},x^{{\color{black}\ell}}_{i_0+2}\}\times\{y^{{\color{black}\ell}}_{j_0-3},y^{{\color{black}\ell}}_{j_0-2},y^{{\color{black}\ell}}_{j_0-1},y^{{\color{black}\ell}}_{j_0},y^{{\color{black}\ell}}_{j_0+1},y^{{\color{black}\ell}}_{j_0+2}\}.$$
We calculate the evaluation at the point $\xe\in\{(i_0-1/2,j_0-1/2)h_{{\color{black}\ell}},(i_0,j_0-1/2)h_{{\color{black}\ell}},(i_0-1/2,j_0)h_{{\color{black}\ell}}\} $ of the polynomials:
$$p^{3}_{\mathbf{j}_1}(\xe), \quad  \mathbf{j}_1=(j^{(1)}_1,j^{(1)}_2) \in \{0,1,2\}^2,$$
being the stencils used for constructing each polynomial $\mathcal{S}^{3}_{(l_1,l_2)}$, with $l_j=0,1,2$, and $j=1,2$, see Figure \ref{wenonormal}, {\color{black} i.e., to compute $p^{3}_{(0,0)}$ we will use
the points $$\mathcal{S}^{4}_{(0,0)}=\{x^{{\ell}}_{i_0-3},x^{{\ell}}_{i_0-2},x^{{\ell}}_{i_0-1},x^{{\ell}}_{i_0}\}\times\{y^{{\ell}}_{j_0-3},y^{{\ell}}_{j_0-2},y^{{\ell}}_{j_0-1},y^{{\ell}}_{j_0}\},$$
for $p^{3}_{(1,0)}$
$$\mathcal{S}^{4}_{(1,0)}=\{x^{{\ell}}_{i_0-2},x^{{\ell}}_{i_0-1},x^{{\ell}}_{i_0},x^{{\ell}}_{i_0+1}\}\times\{y^{{\ell}}_{j_0-3},y^{{\ell}}_{j_0-2},y^{{\ell}}_{j_0-1},y^{{\ell}}_{j_0}\},$$
for $p^{3}_{(0,1)}$
$$\mathcal{S}^{4}_{(0,1)}=\{x^{{\ell}}_{i_0-3},x^{{\ell}}_{i_0-2},x^{{\ell}}_{i_0-1},x^{{\ell}}_{i_0}\}\times\{y^{{\ell}}_{j_0-2},y^{{\ell}}_{j_0-1},y^{{\ell}}_{j_0},y^{{\ell}}_{j_0+1}\},$$
for $p^{3}_{(1,1)}$
$$\mathcal{S}^{4}_{(1,1)}=\{x^{{\ell}}_{i_0-2},x^{{\ell}}_{i_0-1},x^{{\ell}}_{i_0},x^{{\ell}}_{i_0+1}\}\times\{y^{{\ell}}_{j_0-2},y^{{\ell}}_{j_0-1},y^{{\ell}}_{j_0},y^{{\ell}}_{j_0+1}\},$$
and similarly for $p^3_{(l_1,l_2)}$ with $l_j=1,2$ and $j=1,2$. We remark that the stencil used to calculate $p^4_{(0,0)}$ is the union of the stencils above mentioned,  $\mathcal{S}^{4}_{(0,0)}$, $\mathcal{S}^{4}_{(1,0)}$, $\mathcal{S}^{4}_{(0,1)}$, ,$\mathcal{S}^{4}_{(1,1)}$.
 }
\begin{figure}[!hbtp]
\usetikzlibrary{arrows}
\begin{center}
  \input{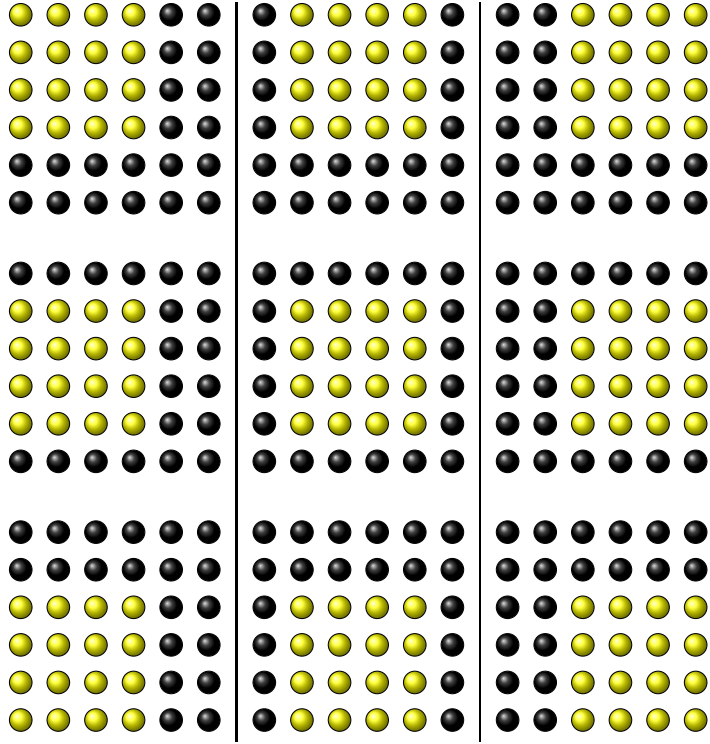}
\end{center}
\caption{Stencils used to get $p^{3}_{\mathbf{j}_1}(\xe), \,  \mathbf{j}_1\in \{0,1,2\}^2$. {\color{black} Each distinct color identifies one of the nine candidate $4 \times 4$ stencils spatially shifted within the global $6 \times 6$ domain, following the classical bivariate WENO framework \cite{arandigamuletrenau}.}}\label{wenonormal}
\end{figure} 
Therefore, by means of the Aitken-Neville formula, we represent $p^4_{\mathbf{j}_0}$, $\mathbf{j}_0=(j^{(0)}_1,j^{(0)}_2) \in\{0,1\}^2$, whose stencils are displayed in Figure \ref{progresivo1}, as the sum of the evaluations of polynomials of degree 3 (Figure \ref{progresive2})
\begin{equation}\label{eqp4primerejemplo}
\begin{split}
&p^{4}_{(0,0)}(\xe)=\sum_{\mathbf{j}_1\in\{0,1\}\times\{0,1\}}C_{(0,0),\mathbf{j}_1}^3(\xe)p^{3}_{\mathbf{j}_1}(\xe),\quad
p^{4}_{(1,0)}(\xe)=\sum_{\mathbf{j}_1\in\{1,2\}\times\{0,1\}}C_{(1,0),\mathbf{j}_1}^3(\xe)p^{3}_{\mathbf{j}_1}(\xe),\\
&p^{4}_{(0,1)}(\xe)=\sum_{\mathbf{j}_1\in\{0,1\}\times\{1,2\}}C_{(0,1),\mathbf{j}_1}^3(\xe)p^{3}_{\mathbf{j}_1}(\xe),\quad p^{4}_{(1,1)}(\xe)=\sum_{\mathbf{j}_1\in\{1,2\}\times\{1,2\}}C_{(1,1),\mathbf{j}_1}^3(\xe)p^{3}_{\mathbf{j}_1}(\xe),\\
\end{split}
\end{equation}
with
$$C^3_{\mathbf{j}_0,\mathbf{j}_1}=C^3_{(j^{(0)}_1,j^{(0)}_2),(j^{(1)}_1,j^{(1)}_2)}(\xe)=C^3_{j^{(0)}_1,j^{(1)}_1}(x^*)C^3_{j^{(0)}_2,j^{(1)}_2}(y^*), \,\, \mathbf{j}_0\in\{0,1\}^2, \,\,
\mathbf{j}_1\in \mathbf{j}_0+\{0,1\}^2,$$
where $C^3_{k,k_1}$, $k=0,1$, $k_1=k+\{0,1\}$ are (see \cite{ARSY20})
\begin{equation}\label{pesos3}
C^3_{0,0}(x^*)=\frac{x^*-x^{{\color{black}\ell}}_{i_0+1}}{x^{{\color{black}\ell}}_{i_0-3}-x^{{\color{black}\ell}}_{i_0+1}}, \quad C^3_{0,1}(x^*)=1-C^3_{0,0}(x^*), \quad C^3_{1,1}(x^*)=\frac{x^*-x^{{\color{black}\ell}}_{i_0+2}}{x^{{\color{black}\ell}}_{i_0-2}-{\color{black}x}^{{\color{black}\ell}}_{i_0+2}}, \quad C^3_{1,2}(x^*)=1-C^3_{1,1}(x^*).
\end{equation}
Now, we replace in Eq. \eqref{eqp4primerejemplo} the linear-weights for non-linear ones:
\begin{equation*}
\tilde{\omega}_{\mathbf{j}_0,\mathbf{j}_1}^3(\xe)=\frac{\tilde{\alpha}_{\mathbf{j}_0,\mathbf{j}_1}^3(\xe)}{\sum_{\mathbf{l}\in\mathbf{j}_0+\{0,1\}^2}\tilde{\alpha}_{\mathbf{j}_0,\mathbf{l}}^3(\xe)},\,\,  \textrm{ where }\,\, \tilde{\alpha}_{\mathbf{j}_0,\mathbf{j}_1}^3(\xe)=\frac{C_{\mathbf{j}_0,\mathbf{j}_1}^3(\xe)}{(\epsilon+I_{\mathbf{j}_0,\mathbf{j}_1}^3)^t},\quad \mathbf{j}_0\in\{0,1\}^2, \, \mathbf{j}_1\in\mathbf{j}_0+\{0,1\}^2,
\end{equation*}
and
$$I_{\mathbf{j}_0,\mathbf{j}_1}^3=I_{\mathbf{j}_1}^3,\,\, \mathbf{j}_0\in\{0,1\}^2, \,\,
\mathbf{j}_1\in \mathbf{j}_0+\{0,1\}^2,$$
being $I^3_{\mathbf{j}_1}$ smoothness indicators which satisfy the properties \ref{P1sm}, \ref{P2sm} and \ref{P3sm}.
We obtain the approximation:
\begin{equation}\label{eqp4primerejemplonolineal}
\begin{split}
&\tilde{p}^{4}_{(0,0)}(\xe)=\sum_{\mathbf{j}_1\in\{0,1\}\times\{0,1\}}\tilde{\omega}_{(0,0),\mathbf{j}_1}^3(\xe)p^{3}_{\mathbf{j}_1}(\xe),\quad
\tilde{p}^{4}_{(1,0)}(\xe)=\sum_{\mathbf{j}_1\in\{1,2\}\times\{0,1\}}\tilde{\omega}_{(1,0),\mathbf{j}_1}^3(\xe)p^{3}_{\mathbf{j}_1}(\xe),\\
&\tilde{p}^{4}_{(0,1)}(\xe)=\sum_{\mathbf{j}_1\in\{0,1\}\times\{1,2\}}\tilde{\omega}_{(0,1),\mathbf{j}_1}^3(\xe)p^{3}_{\mathbf{j}_1}(\xe),\quad \tilde{p}^{4}_{(1,1)}(\xe)=\sum_{\mathbf{j}_1\in\{1,2\}\times\{1,2\}}\tilde{\omega}_{(1,1),\mathbf{j}_1}^3(\xe)p^{3}_{\mathbf{j}_1}(\xe).\\
\end{split}
\end{equation}

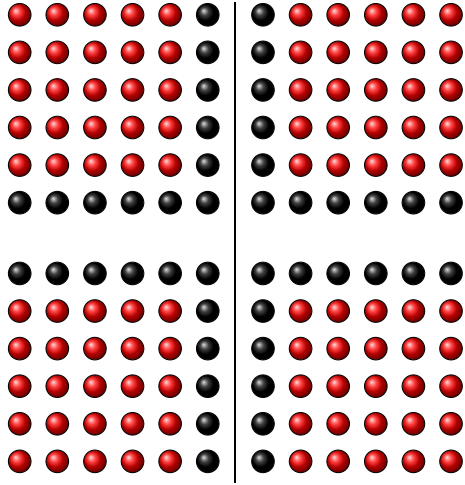
\begin{figure}[H]
\usetikzlibrary{arrows}
\begin{center}
  \begin{tabular}{c|c}
\begin{tikzpicture}
\draw[ball color=black] (0  ,0) circle (.15);
\draw[ball color=black] (0.5,0) circle (.15);
\draw[ball color=black] (1  ,0) circle (.15);
\draw[ball color=black] (1.5,0) circle (.15);
\draw[ball color=black]  (2  ,0) circle (.15);
\draw[ball color=black]  (2.5,0) circle (.15);

\draw[ball color=red] (0  ,0.5) circle (.15);
\draw[ball color=red] (0.5,0.5) circle (.15);
\draw[ball color=red] (1  ,0.5) circle (.15);
\draw[ball color=red] (1.5,0.5) circle (.15);
\draw[ball color=red]  (2  ,0.5) circle (.15);
\draw[ball color=black]  (2.5,0.5) circle (.15);

\draw[ball color=red](0  ,1) circle (.15);
\draw[ball color=red](0.5,1) circle (.15);
\draw[ball color=red] (1  ,1) circle (.15);
\draw[ball color=red] (1.5,1) circle (.15);
\draw[ball color=red] (2  ,1) circle (.15);
\draw[ball color=black] (2.5,1) circle (.15);

\draw[ball color=red](0  ,1.5) circle (.15);
\draw[ball color=red](0.5,1.5) circle (.15);
\draw[ball color=red] (1  ,1.5) circle (.15);
\draw[ball color=red] (1.5,1.5) circle (.15);
\draw[ball color=red] (2  ,1.5) circle (.15);
\draw[ball color=black] (2.5,1.5) circle (.15);

\draw[ball color=red] (0  ,2) circle (.15);
\draw[ball color=red] (0.5,2) circle (.15);
\draw[ball color=red] (1  ,2) circle (.15);
\draw[ball color=red] (1.5,2) circle (.15);
\draw[ball color=red] (2  ,2) circle (.15);
\draw[ball color=black] (2.5,2) circle (.15);

\draw[ball color=red] (0  ,2.5) circle (.15);
\draw[ball color=red] (0.5,2.5) circle (.15);
\draw[ball color=red] (1  ,2.5) circle (.15);
\draw[ball color=red] (1.5,2.5) circle (.15);
\draw[ball color=red] (2  ,2.5) circle (.15);
\draw[ball color=black] (2.5,2.5) circle (.15);
\end{tikzpicture}
&
\begin{tikzpicture}
\draw[ball color=black] (0  ,0) circle (.15);
\draw[ball color=black] (0.5,0) circle (.15);
\draw[ball color=black] (1  ,0) circle (.15);
\draw[ball color=black] (1.5,0) circle (.15);
\draw[ball color=black]  (2  ,0) circle (.15);
\draw[ball color=black]  (2.5,0) circle (.15);

\draw[ball color=black] (0  ,0.5) circle (.15);
\draw[ball color=red] (0.5,0.5) circle (.15);
\draw[ball color=red] (1  ,0.5) circle (.15);
\draw[ball color=red] (1.5,0.5) circle (.15);
\draw[ball color=red]  (2  ,0.5) circle (.15);
\draw[ball color=red]  (2.5,0.5) circle (.15);

\draw[ball color=black](0  ,1) circle (.15);
\draw[ball color=red](0.5,1) circle (.15);
\draw[ball color=red] (1  ,1) circle (.15);
\draw[ball color=red] (1.5,1) circle (.15);
\draw[ball color=red] (2  ,1) circle (.15);
\draw[ball color=red] (2.5,1) circle (.15);

\draw[ball color=black](0  ,1.5) circle (.15);
\draw[ball color=red](0.5,1.5) circle (.15);
\draw[ball color=red] (1  ,1.5) circle (.15);
\draw[ball color=red] (1.5,1.5) circle (.15);
\draw[ball color=red] (2  ,1.5) circle (.15);
\draw[ball color=red] (2.5,1.5) circle (.15);

\draw[ball color=black] (0  ,2) circle (.15);
\draw[ball color=red] (0.5,2) circle (.15);
\draw[ball color=red] (1  ,2) circle (.15);
\draw[ball color=red] (1.5,2) circle (.15);
\draw[ball color=red] (2  ,2) circle (.15);
\draw[ball color=red] (2.5,2) circle (.15);

\draw[ball color=black] (0  ,2.5) circle (.15);
\draw[ball color=red] (0.5,2.5) circle (.15);
\draw[ball color=red] (1  ,2.5) circle (.15);
\draw[ball color=red] (1.5,2.5) circle (.15);
\draw[ball color=red] (2  ,2.5) circle (.15);
\draw[ball color=red] (2.5,2.5) circle (.15);
\end{tikzpicture}
\\[0.5cm]
\begin{tikzpicture}
\draw[ball color=red] (0  ,0) circle (.15);
\draw[ball color=red] (0.5,0) circle (.15);
\draw[ball color=red] (1  ,0) circle (.15);
\draw[ball color=red] (1.5,0) circle (.15);
\draw[ball color=red]  (2  ,0) circle (.15);
\draw[ball color=black]  (2.5,0) circle (.15);

\draw[ball color=red] (0  ,0.5) circle (.15);
\draw[ball color=red] (0.5,0.5) circle (.15);
\draw[ball color=red] (1  ,0.5) circle (.15);
\draw[ball color=red] (1.5,0.5) circle (.15);
\draw[ball color=red]  (2  ,0.5) circle (.15);
\draw[ball color=black]  (2.5,0.5) circle (.15);

\draw[ball color=red](0  ,1) circle (.15);
\draw[ball color=red](0.5,1) circle (.15);
\draw[ball color=red] (1  ,1) circle (.15);
\draw[ball color=red] (1.5,1) circle (.15);
\draw[ball color=red] (2  ,1) circle (.15);
\draw[ball color=black] (2.5,1) circle (.15);

\draw[ball color=red](0  ,1.5) circle (.15);
\draw[ball color=red](0.5,1.5) circle (.15);
\draw[ball color=red] (1  ,1.5) circle (.15);
\draw[ball color=red] (1.5,1.5) circle (.15);
\draw[ball color=red] (2  ,1.5) circle (.15);
\draw[ball color=black] (2.5,1.5) circle (.15);

\draw[ball color=red] (0  ,2) circle (.15);
\draw[ball color=red] (0.5,2) circle (.15);
\draw[ball color=red] (1  ,2) circle (.15);
\draw[ball color=red] (1.5,2) circle (.15);
\draw[ball color=red] (2  ,2) circle (.15);
\draw[ball color=black] (2.5,2) circle (.15);

\draw[ball color=black] (0  ,2.5) circle (.15);
\draw[ball color=black] (0.5,2.5) circle (.15);
\draw[ball color=black] (1  ,2.5) circle (.15);
\draw[ball color=black] (1.5,2.5) circle (.15);
\draw[ball color=black] (2  ,2.5) circle (.15);
\draw[ball color=black] (2.5,2.5) circle (.15);
\end{tikzpicture}
&
\begin{tikzpicture}
\draw[ball color=black] (0  ,0) circle (.15);
\draw[ball color=red] (0.5,0) circle (.15);
\draw[ball color=red] (1  ,0) circle (.15);
\draw[ball color=red] (1.5,0) circle (.15);
\draw[ball color=red]  (2  ,0) circle (.15);
\draw[ball color=red]  (2.5,0) circle (.15);

\draw[ball color=black] (0  ,0.5) circle (.15);
\draw[ball color=red] (0.5,0.5) circle (.15);
\draw[ball color=red] (1  ,0.5) circle (.15);
\draw[ball color=red] (1.5,0.5) circle (.15);
\draw[ball color=red]  (2  ,0.5) circle (.15);
\draw[ball color=red]  (2.5,0.5) circle (.15);

\draw[ball color=black](0  ,1) circle (.15);
\draw[ball color=red](0.5,1) circle (.15);
\draw[ball color=red] (1  ,1) circle (.15);
\draw[ball color=red] (1.5,1) circle (.15);
\draw[ball color=red] (2  ,1) circle (.15);
\draw[ball color=red] (2.5,1) circle (.15);

\draw[ball color=black](0  ,1.5) circle (.15);
\draw[ball color=red](0.5,1.5) circle (.15);
\draw[ball color=red] (1  ,1.5) circle (.15);
\draw[ball color=red] (1.5,1.5) circle (.15);
\draw[ball color=red] (2  ,1.5) circle (.15);
\draw[ball color=red] (2.5,1.5) circle (.15);

\draw[ball color=black] (0  ,2) circle (.15);
\draw[ball color=red] (0.5,2) circle (.15);
\draw[ball color=red] (1  ,2) circle (.15);
\draw[ball color=red] (1.5,2) circle (.15);
\draw[ball color=red] (2  ,2) circle (.15);
\draw[ball color=red] (2.5,2) circle (.15);

\draw[ball color=black] (0  ,2.5) circle (.15);
\draw[ball color=black] (0.5,2.5) circle (.15);
\draw[ball color=black] (1  ,2.5) circle (.15);
\draw[ball color=black] (1.5,2.5) circle (.15);
\draw[ball color=black] (2  ,2.5) circle (.15);
\draw[ball color=black] (2.5,2.5) circle (.15);
\end{tikzpicture}
\end{tabular}
\end{center}
\caption{Stencils used to get $p^4_{\mathbf{j}_0}(\xe)$, $\mathbf{j}_0\in\{0,1\}^2$. {\color{black} The red and black colors highlight the four distinct intermediate stencils resulting from the first stage of the recursive Aitken–Neville procedure, acting as the transition toward the final degree 5 reconstruction.}}\label{progresivo1}
\end{figure}
{\color{black} Now, we want to get $p^5_{(0,0)}(\xe)$ as a convex combination of the approximations $p^{4}_{(0,0)}(\xe)$, $p^{4}_{(0,1)}(\xe)$, $p^{4}_{(1,0)}(\xe)$, $p^{4}_{(1,1)}(\xe)$. Note that the stencils used are
$$\mathcal{S}^{5}_{(0,0)}=\{x^{{\ell}}_{i_0-3},x^{{\ell}}_{i_0-2},x^{{\ell}}_{i_0-1},x^{{\ell}}_{i_0},x^{{\ell}}_{i_0+1}\}\times\{y^{{\ell}}_{j_0-3},y^{{\ell}}_{j_0-2},y^{{\ell}}_{j_0-1},y^{{\ell}}_{j_0},y^{{\ell}}_{j_0+1}\}$$
for $p^{4}_{(0,0)}$;
$$\mathcal{S}^{5}_{(1,0)}=\{x^{{\ell}}_{i_0-2},x^{{\ell}}_{i_0-1},x^{{\ell}}_{i_0},x^{{\ell}}_{i_0+1},x^{{\ell}}_{i_0+2},\}\times\{y^{{\ell}}_{j_0-3},y^{{\ell}}_{j_0-2},y^{{\ell}}_{j_0-1},y^{{\ell}}_{j_0},y^{{\ell}}_{j_0+1}\}$$
for $p^{4}_{(1,0)}$;
$$\mathcal{S}^{5}_{(0,1)}=\{x^{{\ell}}_{i_0-3},x^{{\ell}}_{i_0-2},x^{{\ell}}_{i_0-1},x^{{\ell}}_{i_0},x^{{\ell}}_{i_0+1}\}\times\{y^{{\ell}}_{j_0-2},y^{{\ell}}_{j_0-1},y^{{\ell}}_{j_0},y^{{\ell}}_{j_0+1},y^{{\ell}}_{j_0+2}\}$$
for $p^{4}_{(0,1)}$;
$$\mathcal{S}^{5}_{(1,1)}=\{x^{{\ell}}_{i_0-2},x^{{\ell}}_{i_0-1},x^{{\ell}}_{i_0},x^{{\ell}}_{i_0+1},x^{{\ell}}_{i_0+2}\}\times\{y^{{\ell}}_{j_0-2},y^{{\ell}}_{j_0-1},y^{{\ell}}_{j_0},y^{{\ell}}_{j_0+1},y^{{\ell}}_{j_0+2}\},$$
for $p^4_{(1,1)}$.
Again, the stencil $\mathcal{S}^6_0$ used to compute $p^5_{(0,0)}$ is the union of these stencils. Thus,}
in the final stage of our algorithm for $r=3$, the polynomial of degree 5 is expanded in the following manner:
\begin{equation}\label{p5lineal}
p^{5}_{(0,0)}(\xe)=\sum_{\mathbf{j}_0\in\{0,1\}^2}C_{(0,0),\mathbf{j}_0}^4({\xe})p^{4}_{\mathbf{j}_0}(\xe),
\end{equation}
with
\begin{equation*}
\begin{split}
C^4_{(0,0),(0,0)}(\xe)&=C^4_{0,0}(x^*)C^4_{0,0}(y^*),\quad C^4_{(0,0),(1,0)}(\xe)=C^4_{0,1}(x^*)C^4_{0,0}(y^*),\\
C^4_{(0,0),(0,1)}(\xe)&=C^4_{0,0}(x^*)C^4_{0,1}(y^*),\quad C^4_{(0,0),(1,1)}(\xe)=C^4_{0,1}(x^*)C^4_{0,1}(y^*),
\end{split}
\end{equation*}
where
\begin{equation}\label{pesos4}
C^4_{0,0}(x^*)=\frac{x^*-x^{{\color{black}\ell}}_{i_0+2}}{x^{{\color{black}\ell}}_{i_0-3}-x^{{\color{black}\ell}}_{i_0+2}}, \quad C^4_{0,1}(x^*)=1-C^4_{0,0}(x^*),
\end{equation}
and to establish the new approximation by modifying both the non-linear weights and the polynomial approximation of degree 4
\begin{equation}\label{p5lineal2}
\tilde{\mathcal{I}}^5(\xe;f):=\tilde{p}^{5}_{(0,0)}(\xe)=\sum_{\mathbf{j}_0\in\{0,1\}^2}\tilde{\omega}_{(0,0),\mathbf{j}_0}^4({\xe})\tilde{p}^{4}_{\mathbf{j}_0}(\xe),
\end{equation}
being $\tilde{p}^{4}_{\mathbf{j}_0}(\xe)$ obtained in Eq. \eqref{eqp4primerejemplonolineal} and
\begin{equation*}
\tilde{\omega}_{(0,0),\mathbf{j}_0}^4(\xe)=\frac{\alpha_{(0,0),\mathbf{j}_0}^4(\xe)}{\sum_{\mathbf{l}\in\{0,1\}^2}\alpha_{(0,0),\mathbf{l}}^4(\xe)},\,\,  \textrm{ where }\,\, \alpha_{(0,0),\mathbf{j}_0}^4(\xe)=\frac{C_{(0,0),\mathbf{j}_0}^4(\xe)}{(\epsilon+I_{(0,0),\mathbf{j}_0}^4)^t}, \quad \mathbf{j}_0\in\{0,1\}^2,
\end{equation*}
with the smoothness indicators defined
as:
\begin{equation*}
I_{(0,0),(0,0)}^4=I^3_{(0,0)},\quad I_{(0,0),(0,1)}^4=I^3_{(0,2)},\quad I_{(0,0),(1,0)}^4=I^3_{(2,0)},\quad I_{(0,0),(1,1)}^4=I^3_{(2,2)},
\end{equation*}
where $I^3_{\mathbf{j}_0}, \, \mathbf{j}_0\in\{0,1,2\}^2$ are smoothness indicators satisfying the properties \ref{P1sm}, \ref{P2sm}, and \ref{P3sm}.
\begin{figure}[!hbtp]
\usetikzlibrary{arrows}
\begin{center}
  \input{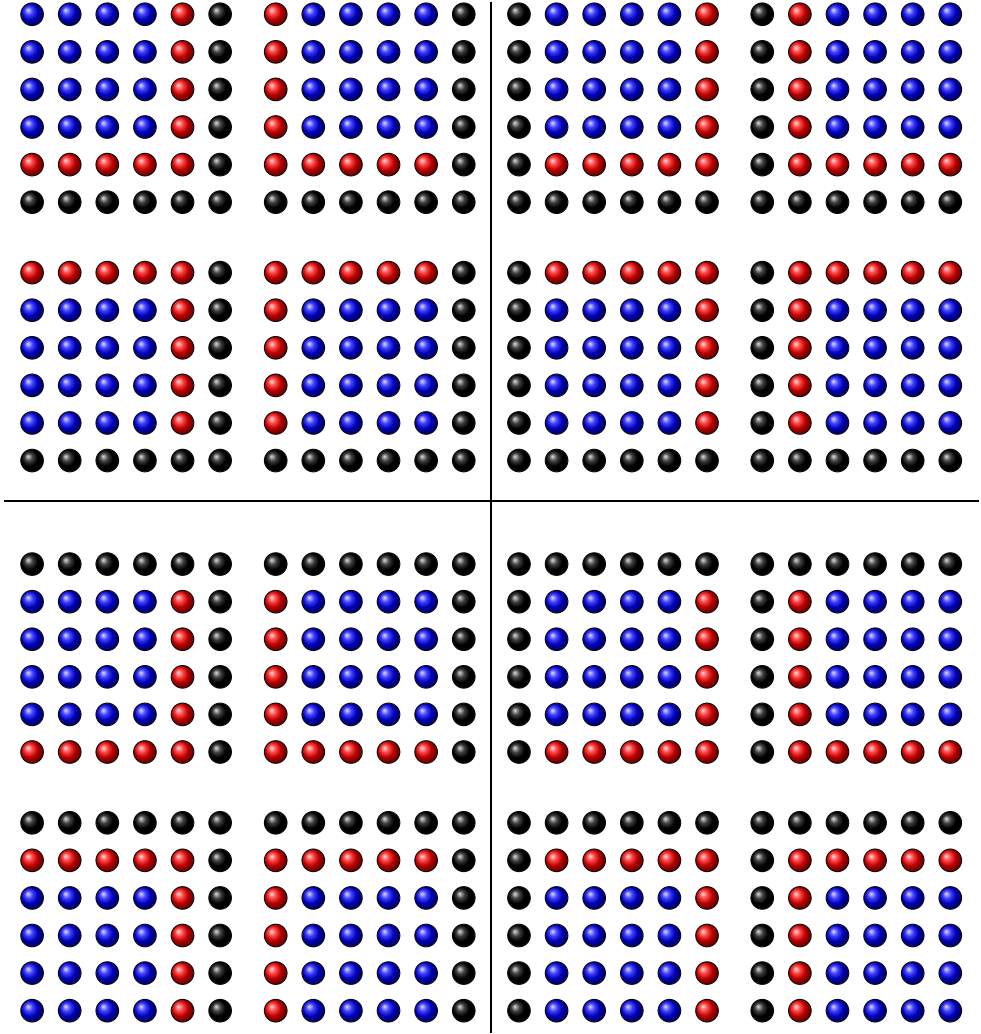}
\end{center}
\caption{Contribution of each stencil of $p^{3}_{\mathbf{j}_1}(\xe)$, $\mathbf{j}_1\in \mathbf{j}_1+\{0,1\}^2$ to the approximation of $p^4_{\mathbf{j}_0}(\xe)$, $\mathbf{j}_0\in\{0,1\}^2$. {\color{black} The red and blue colors illustrate how four specific degree 3 stencils (shown in Figure \ref{wenonormal}) are recursively combined and overlapped to construct a single degree 4 stencil (shown in Figure \ref{progresivo1}).}}\label{progresive2}
\end{figure}
Now, we determine the iterative process  in the general case, \cite{MRSY24}:
\begin{equation*}
\begin{split}
&  \tilde{p}^{r}_{\mathbf{j}_{r-2}}(\xe)=p^{r}_{\mathbf{j}_{r-2}}(\xe), \quad \mathbf{j}_{r-2}=\{0,\hdots,r-1\}^2,\\
&  \tilde{p}^{{{\color{black} \ell}}+1}_{\mathbf{j}_{2r-{{\color{black} \ell}}-3}}(\xe)=\sum_{\mathbf{j}_{2r-{{\color{black} \ell}}-2}\in\mathbf{j}_{2r-{{\color{black} \ell}}-3}+\{0,1\}^2}
  \tilde{\omega}^{{{\color{black} \ell}}}_{\mathbf{j}_{2r-{{\color{black} \ell}}-3},\mathbf{j}_{2r-{{\color{black} \ell}}-2}}(x^*)\tilde{p}_{\mathbf{j}_{2r-{{\color{black} \ell}}-2}}^{{{\color{black} \ell}}}(x^*), \quad
    {{\color{black} \ell}}=r,\dots,2r-2, \, \, \mathbf{j}_{2r-{{\color{black} \ell}}-3}\in\{0,\dots,2r-2-{{\color{black} \ell}}\}^2,\\
\end{split}
\end{equation*}
for ${{\color{black} \ell}}=r,\dots,2r-2$, and $\mathbf{k} \in \{0,\hdots,2r-2-{{\color{black} \ell}}\}^2$, and $\mathbf{k}_1\in\mathbf{k}+\{0,1\}^2$:
\begin{equation}\label{pesosr2d}
\begin{split}
&\tilde{\omega}^{{\color{black} \ell}}_{\mathbf{k},\mathbf{k}_1}(\xe)=\left(\frac{\tilde{\alpha}_{\mathbf{k},\mathbf{k}_1}^{{\color{black} \ell}}}{\tilde{\alpha}_{\mathbf{k},\mathbf{k}}^{{\color{black} \ell}}+\tilde{\alpha}_{\mathbf{k},\mathbf{k}+(1,0)}^{{\color{black} \ell}}+\tilde{\alpha}_{\mathbf{k},\mathbf{k}+(0,1)}^{{\color{black} \ell}}+\tilde{\alpha}_{\mathbf{k},\mathbf{k}+(1,1)}^{{\color{black} \ell}}}\right)(\xe),\\ &\tilde{\alpha}_{\mathbf{k},\mathbf{k}_1}^{{\color{black} \ell}}(\xe)=\frac{{C}_{\mathbf{k},\mathbf{k}_1}^{{\color{black} \ell}}(\xe)}{(\epsilon+I^{{\color{black} \ell}}_{\mathbf{k},\mathbf{k}_1})^t}, \quad \mathbf{k}_1\in \mathbf{k}+\{0,1\}^2,
\end{split}
\end{equation}
where $ I^{{\color{black} \ell}}_{\mathbf{k},\mathbf{k}_1}$ are the smoothness indicators determined by the following formula:
\begin{equation}\label{indices}
{I}^{{\color{black} \ell}}_{\mathbf{k},\mathbf{k}_1}=\left\{
                       \begin{array}{ll}
                         I^r_{\mathbf{k}}, & \hbox{if } \mathbf{k}_1=\mathbf{k}, \\
                         I^r_{\mathbf{k}+({{\color{black} \ell}}-(r-1),0)}, & \hbox{if } \mathbf{k}_1=\mathbf{k}+(1,0),\\
                         I^r_{\mathbf{k}+(0,{{\color{black} \ell}}-(r-1))}, & \hbox{if } \mathbf{k}_1=\mathbf{k}+(0,1),\\
                         I^r_{\mathbf{k}+({{\color{black} \ell}}-(r-1),{{\color{black} \ell}}-(r-1))}, & \hbox{if } \mathbf{k}_1=\mathbf{k}+(1,1),
                       \end{array}
                     \right.
\end{equation}
being $I^r_{\mathbf{k}}$, with $0\leq k_1,k_2\leq r-1$, smoothness indicators satisfying the properties \ref{P1sm}, \ref{P2sm}, and \ref{P3sm}.

Therefore, the bivariate progressive WENO approximation is
$$\tilde{\mathcal{I}}^{2r-1}\left(\xe;f\right)=\tilde{p}^{2r-1}_{(0,0)}(\xe).$$

\section{The new non-separable cell-average WENO method for $r=3$} \label{r3}
In this section, we explicit the formulas for the non-linear filters which we will use in our applications. The idea is easily generalizable to any $(2r-1)^2$ cells but to clarify the concepts, we insert this section here. As the cells are equally spaced, the filters are easier to compute. Let us start with the case $r=3$, where 5 data cells are used. We have for one dimension:
$$(\mathcal{P}_{{\color{black}\ell}-1}^{{\color{black}\ell}} \bar{f}^{{\color{black}\ell}})_{2i-1}=\frac{1}{128}[-3, 22, 128, -22, 3] [\bar f^{{{\color{black}\ell}}-1}_{i-2},\bar f^{{{\color{black}\ell}}-1}_{i-1},\bar f^{{{\color{black}\ell}}-1}_{i},\bar f^{{{\color{black}\ell}}-1}_{i+1},\bar f^{{{\color{black}\ell}}-1}_{i+2}]^T=\mathbf{v}_0^{5}\mathbf{\bar f}^{{{\color{black}\ell}}-1}_{i-2:i+2}.$$
If we denote as $\mathbf{w}_0^{5}=\frac{1}{128}[3, -22, 128, 22, -3]$ then
$$(\mathcal{P}_{{{\color{black}\ell}}-1}^{{\color{black}\ell}} \bar{f}^{{{\color{black}\ell}}-1})_{2i}=\mathbf{w}_0^{5}\mathbf{\bar f}^{{{\color{black}\ell}}-1}_{i-2:i+2}.$$
It is clear that for symmetry we can only work with $\mathbf{v}_0^{5}$, the case for $\mathbf{w}_0^{5}$ is similar. Now, it is easy to see that
\begin{equation*}
\begin{split}
\mathbf{v}_0^{5}&=C^4_{0,0}\mathbf{v}^4_{0}+C^4_{0,1}\mathbf{v}^4_{1}=\frac{C^4_{0,0}}{64}[-3,17, 55, -5, 0]+\frac{C^4_{0,1}}{64}[0,5, 73, -17, 3]\\
&=C^4_{0,0}(C^3_{0,0}\mathbf{v}^3_{0}+C^3_{0,1}\mathbf{v}^3_{1})+C^4_{0,1}(C^3_{1,1}\mathbf{v}^3_{1}+C^3_{1,2}\mathbf{v}^3_{2})\\
&=C^4_{0,0}\left(C^3_{0,0}\left[-\frac{1}{8},\frac{1}{2},\frac{5}{8},0,0\right]+C^3_{0,1}\left[0,\frac{1}{8},1,-\frac{1}{8},0\right]\right)+C^4_{0,1}\left(C^3_{1,1}\left[0,\frac{1}{8},1,-\frac{1}{8},0\right]+C^3_{1,2}\left[0,0,\frac{11}{8},-\frac{1}{2},\frac{1}{8}\right]\right).\\
\end{split}
\end{equation*}
This decomposition, which we have carried out in a single dimension, can be generalized to any dimension by means of the Kronecker product. Here, we focus specifically on the case $n=2$, where the use of vector products allows us to obtain the desired filters. We express them in tensor product form to apply the filters in two dimensions, obtaining
 \begin{equation*}
\begin{split}
(\mathcal{P}_{{{\color{black}\ell}}-1}^{{\color{black}\ell}} \bar{f}^{{{\color{black}\ell}}-1})_{2i-1,2j-1}&=\frac{1}{ 16384 }\begin{bmatrix}
9 & -66 & -384 & 66 & 9 \\
-66 & 484 & 2816 & -484 & -66 \\
-384 & 2816 & 16384 & -2816 & -384 \\
66 & -484 & -2816 & 484 & 66 \\
9 & -66 & -384 & 66 & 9
\end{bmatrix}.*
 \begin{bmatrix}
\bar f^{k-1}_{i-2,j-2} & \bar f^{k-1}_{i-2,j-1} & \bar f^{k-1}_{i-2,j} & \bar f^{k-1}_{i-2,j+1} & \bar f^{k-1}_{i-2,j+2} \\
\bar f^{k-1}_{i-1,j-2} & \bar f^{k-1}_{i-1,j-1} & \bar f^{k-1}_{i-1,j} & \bar f^{k-1}_{i-1,j+1} & \bar f^{k-1}_{i-1,j+2} \\
\bar f^{k-1}_{i,j-2} &   \bar f^{k-1}_{i,j-1} & \bar f^{k-1}_{i,j} & \bar f^{k-1}_{i,j+1} &     \bar f^{k-1}_{i,j+2} \\
\bar f^{k-1}_{i+1,j-2} & \bar f^{k-1}_{i+1,j-1} & \bar f^{k-1}_{i+1,j} & \bar f^{k-1}_{i+1,j+1} & \bar f^{k-1}_{i+1,j+2} \\
\bar f^{k-1}_{i+2,j-2} & \bar f^{k-1}_{i+2,j-1} & \bar f^{k-1}_{i+2,j} & \bar f^{k-1}_{i+2,j+1} & \bar f^{k-1}_{i+2,j+2} \\
\end{bmatrix}   \\
&=
(\mathbf{v}_0^{5})^T (\mathbf{v}_0^{5}).*\mathbf{\bar f}^{{{\color{black}\ell}}-1}_{i-2:i+2,j-2:j+2},
\end{split}
\end{equation*}
being the product $.*$ defined for two square matrices $A,B\in \mathbb{R}^{N\times N}$ as
$$A.*B=\sum_{i,j=1}^N A_{i,j}B_{i,j}.$$
At this stage, by employing the two-dimensional Aitken formula, we obtain the decomposition
\begin{equation}\label{eqmatrices}
\begin{split}
(\mathbf{v}_0^{5})^T (\mathbf{v}_0^{5})
=&(C^4_{0,0}\mathbf{v}^4_{0}+C^4_{0,1}\mathbf{v}^4_{1})^T  (C^4_{0,0}\mathbf{v}^4_{0}+C^4_{0,1}\mathbf{v}^4_{1})\\
=&C^4_{0,0}C^4_{0,0}(\mathbf{v}^4_{0})^T   \mathbf{v}^4_{0} + C^4_{0,1}C^4_{0,0}(\mathbf{v}^4_{1})^T  \mathbf{v}^4_{0}+C^4_{0,0}C^4_{0,1}(\mathbf{v}^4_{0})^T \mathbf{v}^4_{1}+C^4_{0,1}C^4_{0,1}(\mathbf{v}^4_{1})^T \mathbf{v}^4_{1}\\
=&C^4_{0,0}C^4_{0,0}((C^3_{0,0}\mathbf{v}^3_{0}+C^3_{0,1}\mathbf{v}^3_{1})^T   (C^3_{0,0}\mathbf{v}^3_{0}+C^3_{0,1}\mathbf{v}^3_{1}))+ C^4_{0,1}C^4_{0,0}((C^3_{1,1}\mathbf{v}^3_{1}+C^3_{1,2}\mathbf{v}^3_{2})^T (C^3_{0,0}\mathbf{v}^3_{0}+C^3_{0,1}\mathbf{v}^3_{1}))\\
&+C^4_{0,0}C^4_{0,1}((C^3_{0,0}\mathbf{v}^3_{0}+C^3_{0,1}\mathbf{v}^3_{1})^T (C^3_{1,1}\mathbf{v}^3_{1}+C^3_{1,2}\mathbf{v}^3_{2}))+C^4_{0,1}C^4_{0,1}(C^3_{1,1}\mathbf{v}^3_{1}+C^3_{1,2}\mathbf{v}^3_{2})^T (C^3_{1,1}\mathbf{v}^3_{1}+C^3_{1,2}\mathbf{v}^3_{2}))\\
=&C^4_{0,0}C^4_{0,0}(C^3_{0,0}C^3_{0,0}(\mathbf{v}^3_{0})^T\mathbf{v}^3_{0}+C^3_{0,1}C^3_{0,0}(\mathbf{v}^3_{1})^T \mathbf{v}^3_{0}+C^3_{0,0}C^3_{0,1}(\mathbf{v}^3_{0})^T \mathbf{v}^3_{1} + C^3_{0,1}C^3_{0,1}(\mathbf{v}^3_{1})^T\mathbf{v}^3_{1})\\
&+ C^4_{0,1}C^4_{0,0}(C^3_{1,1}C^3_{0,0}(\mathbf{v}^3_{1})^T\mathbf{v}^3_{0}+C^3_{0,0}C^3_{1,2}(\mathbf{v}^3_{2})^T\mathbf{v}^3_{0}+C^3_{1,1}C^3_{0,1}(\mathbf{v}^3_{1})^T\mathbf{v}^3_{1}+C^3_{0,1}C^3_{1,2}(\mathbf{v}^3_{2})^T\mathbf{v}^3_{1})\\
&+C^4_{0,0}C^4_{0,1}(C^3_{0,0}C^3_{1,2}(\mathbf{v}^3_{0})^T\mathbf{v}^3_{2}+C^3_{0,1}C^3_{1,1}(\mathbf{v}^3_{1})^T \mathbf{v}^3_{1}+C^3_{0,0}C^3_{1,1}(\mathbf{v}^3_{0})^T\mathbf{v}^3_{1}+C^3_{1,2}C^3_{0,1}(\mathbf{v}^3_{1})^T\mathbf{v}^3_{2})\\
&+C^4_{0,1}C^4_{0,1}(C^3_{1,1}C^3_{1,1}(\mathbf{v}^3_{1})^T\mathbf{v}^3_{1}+C^3_{1,1}C^3_{1,2}(\mathbf{v}^3_{2})^T\mathbf{v}^3_{1} +C^3_{1,1}C^3_{1,2}(\mathbf{v}^3_{1})^T\mathbf{v}^3_{2}+C^3_{1,2}C^3_{1,2}(\mathbf{v}^3_{2})^T\mathbf{v}^3_{2}),\\
\end{split}
\end{equation}
with
$$C^4_{0,0}=C^4_{0,1}=\frac{1}{2},\quad C^3_{0,0}=C^3_{1,2}=\frac{3}{8},\quad C^3_{0,1}=C^3_{1,1}=\frac{5}{8}.$$
\begin{equation*}
(\mathbf{v}^3_{0})^T \mathbf{v}^3_{0}=\frac{1}{64 }\begin{bmatrix}
     1  &  -4 &   -5  &   0  &   0    \\
    -4  &  16 &   20  &   0  &   0    \\
    -5  &  20 &   25  &   0  &   0    \\
     0  &   0 &    0  &   0  &   0    \\
     0  &   0 &    0  &   0  &   0
\end{bmatrix},\quad  (\mathbf{v}^3_{0})^T \mathbf{v}^3_{1}=\frac{1}{64 }\begin{bmatrix}
        0   & -1  &  -8  &   1  &   0  \\
        0   &  4  &  32  &  -4  &   0  \\
        0   &  5  &  40  &  -5  &   0  \\
        0   &  0  &   0  &   0  &   0  \\
        0   &  0  &   0  &   0  &   0
\end{bmatrix},
\end{equation*}
\begin{equation*}
(\mathbf{v}^3_{0})^T \mathbf{v}^3_{2}=\frac{1}{64 }\begin{bmatrix}
     0  &   0 &  -11  &   4 &   -1  \\
     0  &   0 &   44  & -16 &    4  \\
     0  &   0 &   55  & -20 &    5  \\
     0  &   0 &    0  &   0 &    0  \\
     0  &   0 &    0  &   0 &    0
\end{bmatrix}, \quad (\mathbf{v}^3_{1})^T \mathbf{v}^3_{1}=\frac{1}{64 }\begin{bmatrix}
     0 &    0   &  0  &   0 &    0   \\
     0 &    1   &  8  &  -1 &    0   \\
     0 &    8   & 64  &  -8 &    0   \\
     0 &   -1   & -8  &   1 &    0   \\
     0 &    0   &  0  &   0 &    0
\end{bmatrix},
\end{equation*}

\begin{equation*}
(\mathbf{v}^3_{1})^T \mathbf{v}^3_{2}=\frac{1}{64 }\begin{bmatrix}
     0  &   0 &    0  &   0   &  0  \\
     0  &   0 &   11  &  -4   &  1  \\
     0  &   0 &   88  & -32   &  8  \\
     0  &   0 &  -11  &   4   & -1  \\
     0  &   0 &    0  &   0   &  0
\end{bmatrix}, \quad (\mathbf{v}^3_{2})^T \mathbf{v}^3_{2}=\frac{1}{64 }\begin{bmatrix}
          0  &    0  &    0 &     0 &     0     \\
          0  &    0  &    0 &     0 &     0     \\
          0  &    0  &  121 &   -44 &    11     \\
          0  &    0  &  -44 &    16 &    -4     \\
          0  &    0  &   11 &    -4 &     1
\end{bmatrix}.
\end{equation*}
As it can be seen in the above matrices, the goal is to avoid using nodes that are close to a discontinuity. Thus, we replace the linear weights appearing in Eq. \eqref{eqmatrices} with nonlinear weights in the following manner:
\begin{equation}\label{eqmatrices2}
\begin{split}
(\mathbf{v}_0^{5})^T (\mathbf{v}_0^{5})
=&\tilde{\omega}^4_{(0,0),(0,0)}(\tilde{\omega}^3_{(0,0),(0,0)}(\mathbf{v}^3_{0})^T\mathbf{v}^3_{0}+\tilde{\omega}^3_{(0,0),(1,0)}(\mathbf{v}^3_{1})^T \mathbf{v}^3_{0}+\tilde{\omega}^3_{(0,0),(0,1)}(\mathbf{v}^3_{0})^T \mathbf{v}^3_{1} + \tilde{\omega}^3_{(0,0),(1,1)}(\mathbf{v}^3_{1})^T\mathbf{v}^3_{1})\\
&+ \tilde{\omega}^4_{(0,0),(1,0)}(\tilde{\omega}^3_{(1,0),(1,0)}(\mathbf{v}^3_{1})^T\mathbf{v}^3_{0}+\tilde{\omega}^3_{(0,1),(0,2)}(\mathbf{v}^3_{2})^T\mathbf{v}^3_{0}+\tilde{\omega}^3_{(1,0),(1,0)}(\mathbf{v}^3_{1})^T\mathbf{v}^3_{1}+\tilde{\omega}^3_{(0,1),(1,2)}(\mathbf{v}^3_{2})^T\mathbf{v}^3_{1})\\
&+\tilde{\omega}^4_{(0,0),(0,1)}(\tilde{\omega}^3_{(0,1),(0,2)}(\mathbf{v}^3_{0})^T\mathbf{v}^3_{2}+\tilde{\omega}^3_{(0,1),(1,1)}(\mathbf{v}^3_{1})^T \mathbf{v}^3_{1}+\tilde{\omega}^3_{(1,0),(2,0)}(\mathbf{v}^3_{0})^T\mathbf{v}^3_{2}+\tilde{\omega}^3_{(1,0),(2,1)}(\mathbf{v}^3_{1})^T\mathbf{v}^3_{2})\\
&+\tilde{\omega}^4_{(0,0),(1,1)}(\tilde{\omega}^3_{(1,1),(1,1)}(\mathbf{v}^3_{1})^T\mathbf{v}^3_{1}+\tilde{\omega}^3_{(1,1),(1,2)}(\mathbf{v}^3_{2})^T\mathbf{v}^3_{1} +\tilde{\omega}^3_{(1,1),(1,2)}(\mathbf{v}^3_{1})^T\mathbf{v}^3_{2}+\tilde{\omega}^4_{(1,2),(1,2)}(\mathbf{v}^3_{2})^T\mathbf{v}^3_{2}),\\
\end{split}
\end{equation}
with
\begin{equation*}
\tilde{\omega}_{\mathbf{j}_0,\mathbf{j}_1}^3=\frac{\tilde{\alpha}_{\mathbf{j}_0,\mathbf{j}_1}^3}{\sum_{\mathbf{l}\in\mathbf{j}_0+\{0,1\}^2}\tilde{\alpha}_{\mathbf{j}_0,\mathbf{l}}^3},\,\,  \textrm{ where }\,\, \tilde{\alpha}_{\mathbf{j}_0,\mathbf{j}_1}^3=\frac{C_{\mathbf{j}_0,\mathbf{j}_1}^3}{(\epsilon+I_{\mathbf{j}_0,\mathbf{j}_1}^3)^t},\quad \mathbf{j}_0\in\{0,1\}^2, \, \mathbf{j}_1\in\mathbf{j}_0+\{0,1\}^2,
\end{equation*}
where
$$I_{\mathbf{j}_0,\mathbf{j}_1}^3=I_{\mathbf{j}_1}^3,\,\,\text{ with }\,\,\mathbf{j}_1\in\{(0,0),(1,0),(0,1),(1,1),(2,0),(0,2),(1,2),(2,1)\},$$
and\begin{equation*}
\tilde{\omega}_{(0,0),\mathbf{j}_0}^4=\frac{\alpha_{(0,0),\mathbf{j}_0}^4}{\sum_{\mathbf{l}\in\{0,1\}^2}\alpha_{(0,0),\mathbf{l}}^4},\,\,  \textrm{ where }\,\, \alpha_{(0,0),\mathbf{j}_0}^4=\frac{C_{(0,0),\mathbf{j}_0}^4}{(\epsilon+I_{(0,0),\mathbf{j}_0}^4)^t}, \quad \mathbf{j}_0\in\{0,1\}^2,
\end{equation*}
being 
\begin{equation*}
I_{(0,0),(0,0)}^4=I^3_{(0,0)},\quad I_{(0,0),(0,1)}^4=I^3_{(0,2)},\quad I_{(0,0),(1,0)}^4=I^3_{(2,0)},\quad I_{(0,0),(1,1)}^4=I^3_{(2,2)}.
\end{equation*}
Now, to compute the smoothness indicators, we rely on the relations given in Eq.~\eqref{fijcell} and Eq.~\eqref{fijcell2}, which allow us to evaluate them using the formulas presented in \cite{arandigamuletrenau}. Accordingly, the smoothness indicators can be calculated as:
\[ I_{{\color{black} \mathbf{k}}}^3 = \sum_{\substack{(m,n)\in \{0,1,2\}^2\\ (m,n)\neq (0,0)}}
\int_{x_{i-1}}^{x_i}\int_{y_{j-1}}^{y_j} \frac{\partial^{n+m+2}}{\partial x^{m+1}\partial y^{n+1}}\mathcal{I}_{k_1,k_2}(x,y;F^{{{\color{black}\ell}}-1})dxdy.
\]
Being ${\color{black} \mathbf{k}}\in \{0,1,2\}^2$ and $\mathcal{I}_{k_1,k_2}(x,y;F^{{{\color{black}\ell}}-1})$ the corresponding interpolant. We calculate an explicit expression for these smoothness indicators based on linear combinations of the data. Therefore, as we only use these values, the computational cost of the new method is similar to classical WENO. It is very easy to extend this procedure to any degree $r$ but, for simplicity, we center our explanation on $r=3$.

\section{Numerical experiments}\label{expnum}

In this section, we present a series of numerical experiments designed to compare the performance of linear reconstruction with that of the nonlinear reconstruction based on the newly proposed non-separable cell-averaged WENO-$2r$ method for $r=3$ (see Section~\ref{r3}). First, we apply both methods to a set of test functions exhibiting different types of discontinuities, in order to evaluate which approach yields better results. To complement this analysis, we also perform image compression experiments using color digital images, allowing us to observe and assess the compression capabilities of each method in a practical setting.

\subsection{Cell-averaging of functions}

We will analyze and compare the behaviour of linear and non-linear methods when dealing with functions that have discontinuities. Since the data are represented in terms of cell averages, as we mention in the previous section we consider the data as:
\begin{equation} \label{cell}
\bar{f}^{{\color{black}\ell}}_{i,j} = \frac{1}{h_{{\color{black}\ell}}^2} 
\int_{x^{{\color{black}\ell}}_{i-1}}^{x^{{\color{black}\ell}}_{i}}\!\!\int_{y^{{\color{black}\ell}}_{j-1}}^{y^{{\color{black}\ell}}_{j}} 
f(x,y)\,dx\,dy,
\end{equation}
where ${\color{black}\ell}$ is the level of resolution, $h_{{\color{black}\ell}}^2$ denotes the mesh spacing and $\Omega_{i,j}^{{\color{black}\ell}}=[x^{{\color{black}\ell}}_{i-1},x^{{\color{black}\ell}}_{i+1}]\times[y^{{\color{black}\ell}}_{j-1},y^{{\color{black}\ell}}_{j}]$ is the corresponding computational cell. Each $\bar{f}^{{\color{black}\ell}}_{i,j}$ therefore represents the mean value of $f$ over the cell $\Omega^{{\color{black}\ell}}_{i,j}$.

The accuracy of the reconstruction at level ${\color{black}\ell}$ is assessed in the discrete $\ell^2$ norm, defined as
\begin{equation} \label{eq:L2cell}
E_{2} = 
\big\| 
(\mathcal{P}^{{\color{black}\ell}}_{{\color{black}\ell}-1} \bar{f}^{{\color{black}\ell}-1})_{i,j} - \bar{f}^{{\color{black}\ell}}_{i,j}
\big\|_{2}
=
\left(
h_{{\color{black}\ell}}^2 \sum_{i,j=1}^{2^{{\color{black}\ell}}}
\left|
(\mathcal{P}^{{\color{black}\ell}}_{{\color{black}\ell}-1} \bar{f}^{{\color{black}\ell}-1})_{i,j} - \bar{f}^{{\color{black}\ell}}_{i,j}
\right|^2
\right)^{1/2},
\end{equation}
where $\bar{f}^{{\color{black}\ell}}_{i,j}$ denotes the exact cell-averaged data at level ${\color{black}\ell}$ and 
$(\mathcal{P}^{{\color{black}\ell}}_{{\color{black}\ell}-1} \bar{f}^{{\color{black}\ell}-1})_{i,j}$ represents the reconstructed cell averages obtained by prolongation from level ${\color{black}\ell}-1$ to level ${\color{black}\ell}$ using either the linear or the WENO predictor. This cell-averaged $\ell^2$ error provides a global measure of the mean quadratic deviation between the exact and reconstructed fields, being consistent with the conservative formulation of the multiresolution representation.

{\color{black} The} resolution level we will use for decimation and prediction is $L=1$, performing a multiresolution calculation without details in order to compare the linear and non-linear methods, that is, we simply reduce the resolution by one level and make predictions from this level without storing any error.


First of all, we consider the function
\begin{equation}
    g(x,y) = x^{3} - y^{3} + 2.1\,x^{2}y^{2} + x^{2} - 0.1\,y^{2} - y + x - 0.01\,xy + 1,
    \label{eq:gxy}
\end{equation}
defined over the domain $[-1,1]^2$. The numerical experiments are performed on a uniform grid of $512 \times 512$ computational cells. The data are represented in terms of cell averages, Eq. \eqref{cell}.

To introduce a controlled discontinuity in the domain, a modified function is defined by adding a constant in the lower half of the domain (along the line $y=0$):
\begin{equation}
    \tilde{g}(x,y) =
    \begin{cases}
        g(x,y), & y < 0,\\[4pt]
        g(x,y) + C, & y \geq 0,
    \end{cases}
    \label{eq:gtilde}
\end{equation}
where $C=16$ is the chosen order of magnitude jump. This configuration allows the behaviour of the linear and WENO methods to be analysed in smooth regions and in the presence of discontinuities.

\begin{figure}[H]
    \centering
    \begin{tabular}{ccc}
    \includegraphics[width=0.32\textwidth]{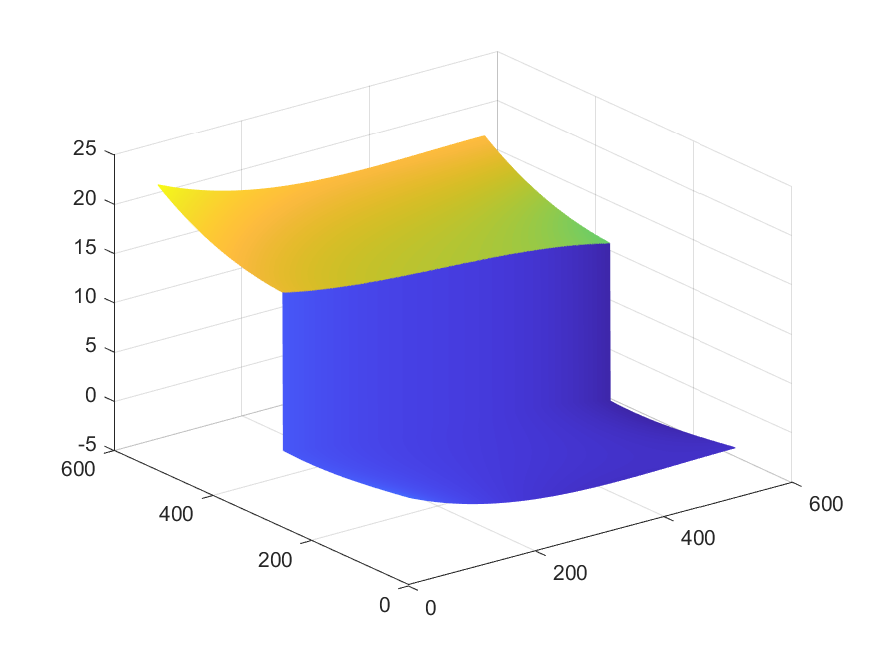}&
    \includegraphics[width=0.32\textwidth]{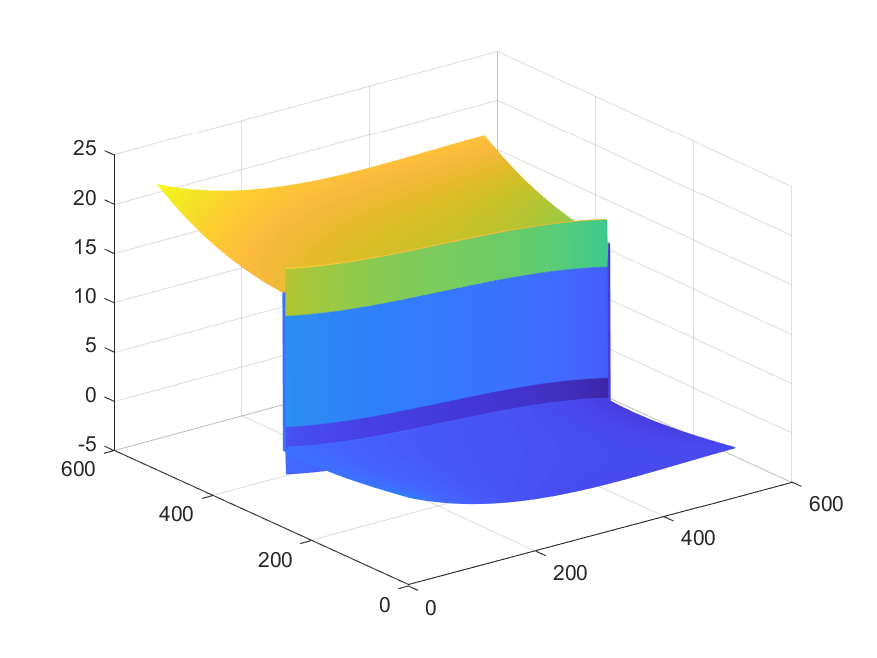}&
    \includegraphics[width=0.32\textwidth]{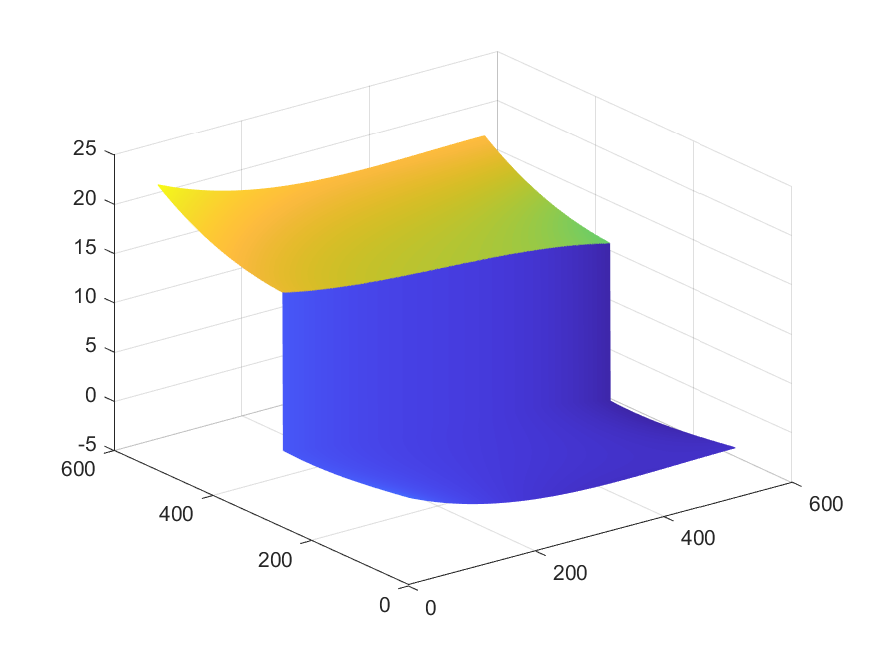}\\
        (a) & (b) & (c)
    \end{tabular}
    \caption{Three-dimensional representations of the polynomial test function \eqref{eq:gtilde} defined on the domain $[-1,1]^2$. Panel \textbf{(a)} shows the original function exhibiting a discontinuity of magnitude $C=16$ in the lower half of the domain. Panels \textbf{(b)} and \textbf{(c)} display the reconstructed surfaces obtained with the linear and WENO-2D prediction schemes, respectively.}
    \label{surface1}
\end{figure}
Figure \ref{surface1} shows the test polynomial function \eqref{eq:gtilde} defined in the square domain $[-1,1]^2$, together with its decimated version and the reconstructions obtained using two prediction methods: the linear method and the new progressive bivariate WENO method. As it can be seen, the linear method adequately reproduces the smooth regions but exhibits noticeable numerical diffusion in the jump area, while the WENO scheme better preserves the shape and height of the discontinuity, avoiding spurious oscillations around the edge. Furthermore, the error in the discrete $\ell^2$ norm, Eq. \eqref{eq:L2cell}, obtained with the linear scheme was $E_2^{\rm LIN}=$ {\color{black}4.2171e-01}, whereas for the WENO method it was $E_2^{\rm WENO}=$ {\color{black}1.8473e-05}. These results confirm the WENO scheme's greater capacity to preserve discontinuities without significant loss of accuracy in smooth regions.

Now, we study the reconstruction of a two-dimensional surface with an embedded discontinuity. We choose the same example as the one presented in \cite{MRSY24}, the function:
\begin{equation} \label{surf2}
h({\color{black}x},{\color{black} y}) =
\begin{cases}
e^{{\color{black} x} + {\color{black} y}} \cos({\color{black} x} - {\color{black} y}), & {\color{black} x} + {\color{black} y} \le 0,\\[1mm]
e^{{\color{black} x} + {\color{black} y}} \cos({\color{black} x} - {\color{black} y}) + 1, & {\color{black} x} + {\color{black} y} > 0.
\end{cases}
\end{equation}
is defined on the square domain $[-1,1]^2$ and combines a smooth exponential-cosine term with a jump along the line ${\color{black} x} + {\color{black} y} = 0$. This setup allows for the assessment of linear and non-linear reconstruction schemes in the presence of discontinuities.

The upper-right triangular portion of the domain exhibits a unit jump, which allows us to evaluate the performance of different reconstruction schemes in handling discontinuities while preserving accuracy in smooth regions.

As shown in Figure \ref{surface2}, the linear reconstruction captures the smooth trends of the surface but significantly smears the discontinuity. In contrast, the WENO-2D reconstruction accurately preserves the jump while maintaining high fidelity in smooth regions. This confirms the superior capability of WENO schemes for surfaces with mixed smooth and discontinuous features.

\begin{figure}[h!]
    \centering
    \begin{tabular}{ccc}
        \includegraphics[width=0.32\textwidth]{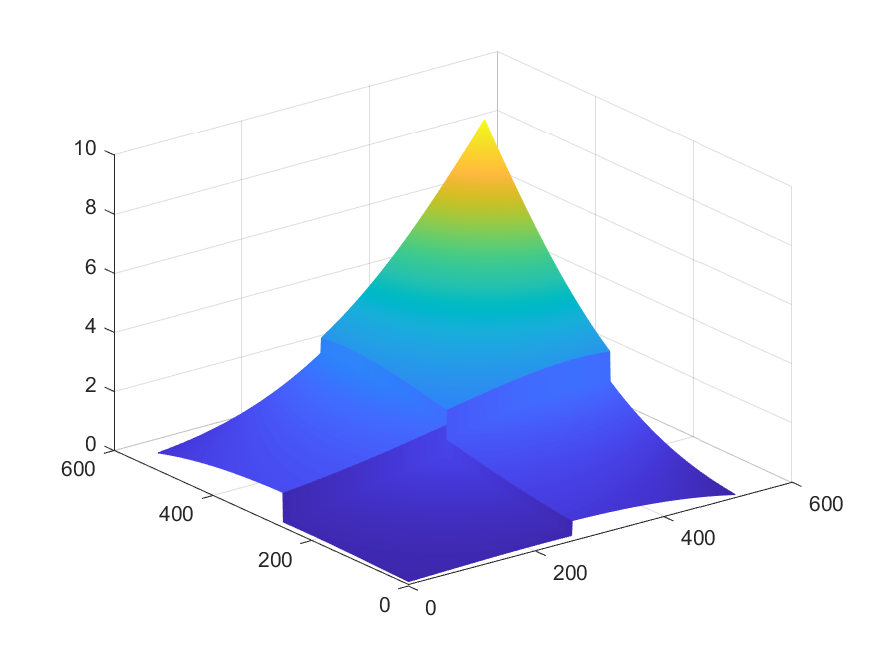} &
        \includegraphics[width=0.32\textwidth]{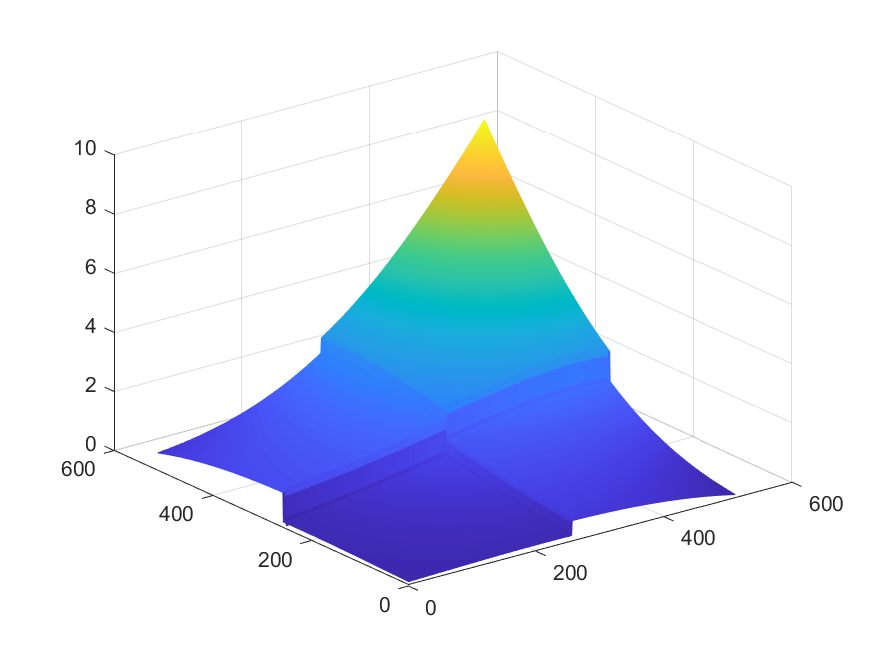} &
        \includegraphics[width=0.32\textwidth]{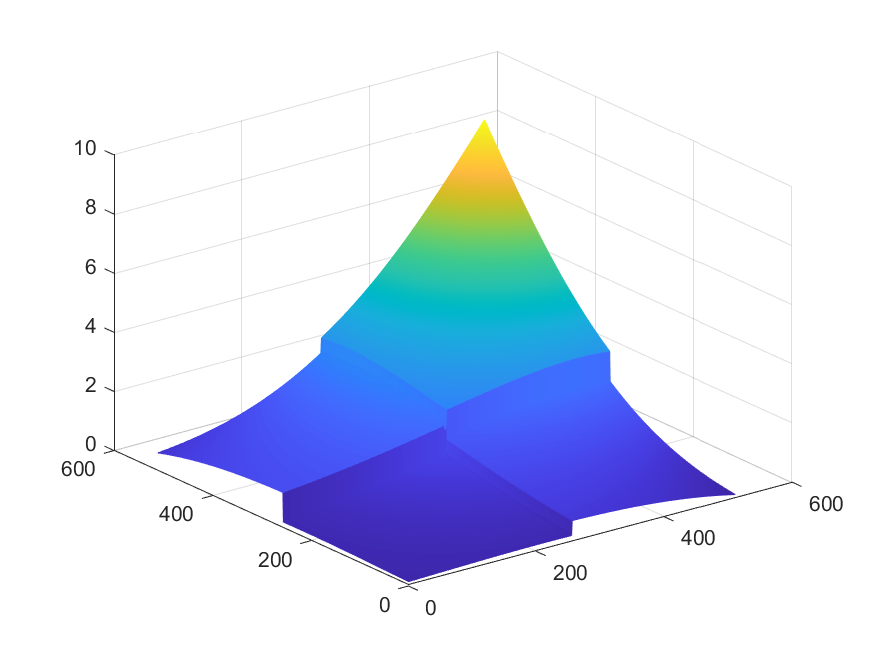} \\
        (a)  & (b)  & (c) 
    \end{tabular}
    \caption{Three-dimensional representations of the surface \eqref{surf2} on $[-1,1]^2$. Panel \textbf{(a)} shows the surface decimated to half resolution. Panel \textbf{(b)} presents the reconstruction obtained with a linear scheme, while panel \textbf{(c)} shows the reconstruction obtained using the progressive WENO-2D method. The WENO reconstruction preserves the discontinuity along ${\color{black} x}+{\color{black} y}=0$, whereas the linear scheme smooths the jump.}
    \label{surface2}
\end{figure}

The errors calculated demonstrate that WENO reconstruction minimizes deviation near discontinuities without compromising accuracy in continuous areas ($E_2^{\rm LIN} =$ {\color{black} 2.9468e-02}, while the WENO reconstruction achieved $E_2^{\rm WENO} =$ {\color{black} 7.8814e-03}).

Finally, we analyze the reconstruction of a classical two-dimensional Franke function \cite{Franke} modified to include a discontinuity along a straight line. The domain is the square $[-1,1]^2$, and the discontinuity can be either horizontal or vertical. This setup allows us to evaluate the performance of linear and non-linear reconstruction schemes in the presence of abrupt changes in the function values.

The modified Franke's function is defined as

\begin{equation} \label{franke_disc}
f({\color{black} x},{\color{black} y}) =
\begin{cases}
f_1({\color{black} x},{\color{black} y}), & \ell({\color{black} x},{\color{black} y})<0,\\[1mm]
1 + f_1({\color{black} x},{\color{black} y}), & \ell({\color{black} x},{\color{black} y})\ge 0,
\end{cases}
\end{equation}
where
\begin{equation*}
f_1({\color{black} x}, {\color{black} y}) =
\frac{3}{4} e^{ -\frac{(9{\color{black} x} - 2)^2}{4} - \frac{(9{\color{black} y} - 2)^2}{4} }
+ \frac{3}{4} e^{ -\frac{(9{\color{black} x} + 1)^2}{49} - \frac{(9{\color{black} y} + 1)}{10} }
+ \frac{1}{2} e^{ -\frac{(9{\color{black} x} - 7)^2}{4} - \frac{(9{\color{black} y} - 3)^2}{4} }
- \frac{1}{5} e^{ -(9{\color{black} x} - 4)^2 - (9{\color{black} y} - 7)^2 }.
\end{equation*}
and $\ell({\color{black} x},{\color{black} y})$ represents the line defining the discontinuity:
\[
\ell({\color{black} x},{\color{black} y}) =
\begin{cases}
{\color{black} y}, & \text{for a horizontal line at } y=0,\\
{\color{black} x}, & \text{for a vertical line at } x=0.
\end{cases}
\]

As observed in Figure~\ref{fig:franke_recon}, the linear reconstruction fails to accurately capture the discontinuity, resulting in a smoothed transition across the jump. In contrast, the WENO reconstruction preserves the sharpness of the discontinuity while maintaining high accuracy in smooth regions. This demonstrates the enhanced capability of non-linear schemes in handling surfaces with localized abrupt changes.

\begin{figure}[h!]
    \centering
    \begin{tabular}{ccc}
    \includegraphics[width=0.32\textwidth]{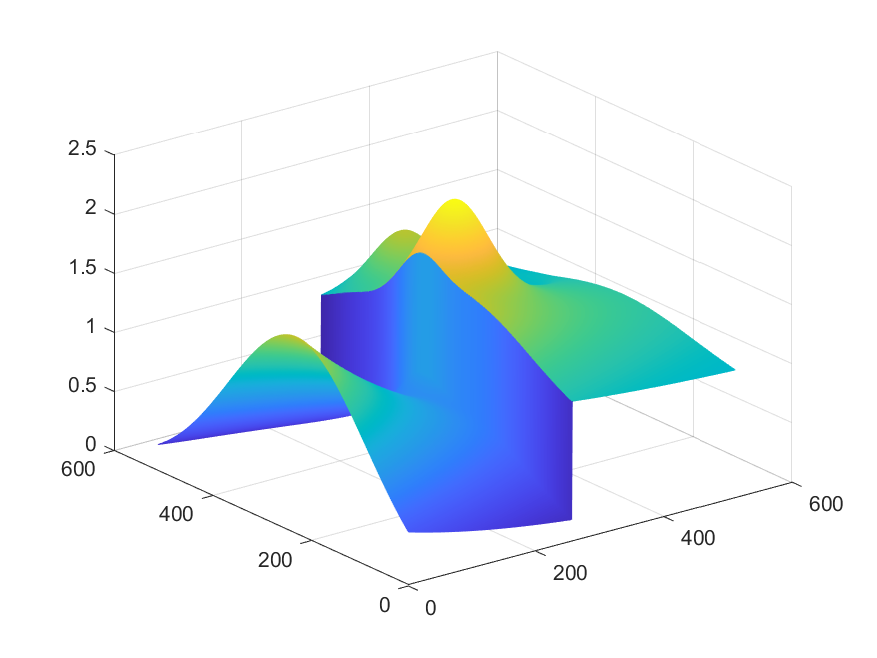} &
    \includegraphics[width=0.32\textwidth]{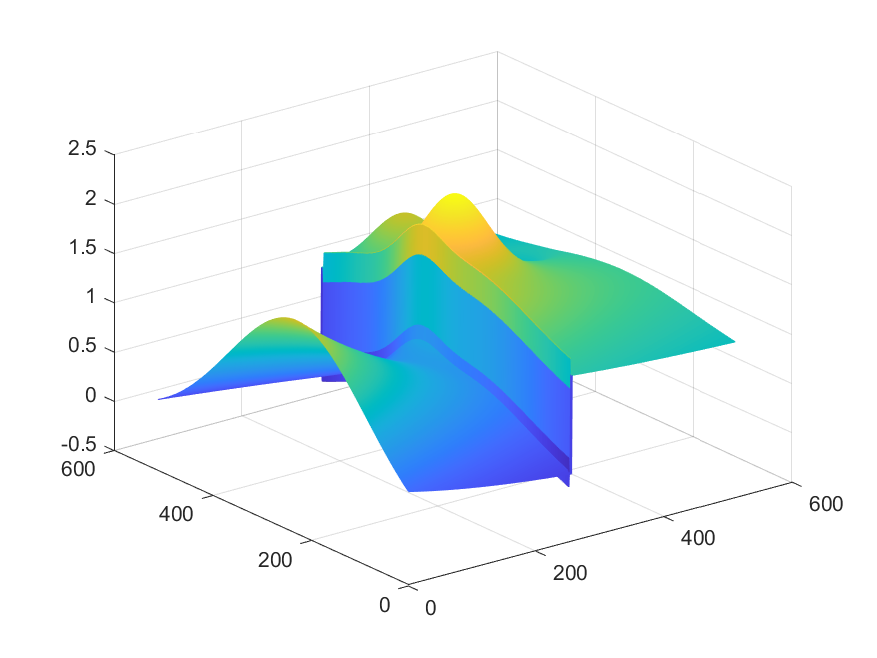} &
    \includegraphics[width=0.32\textwidth]{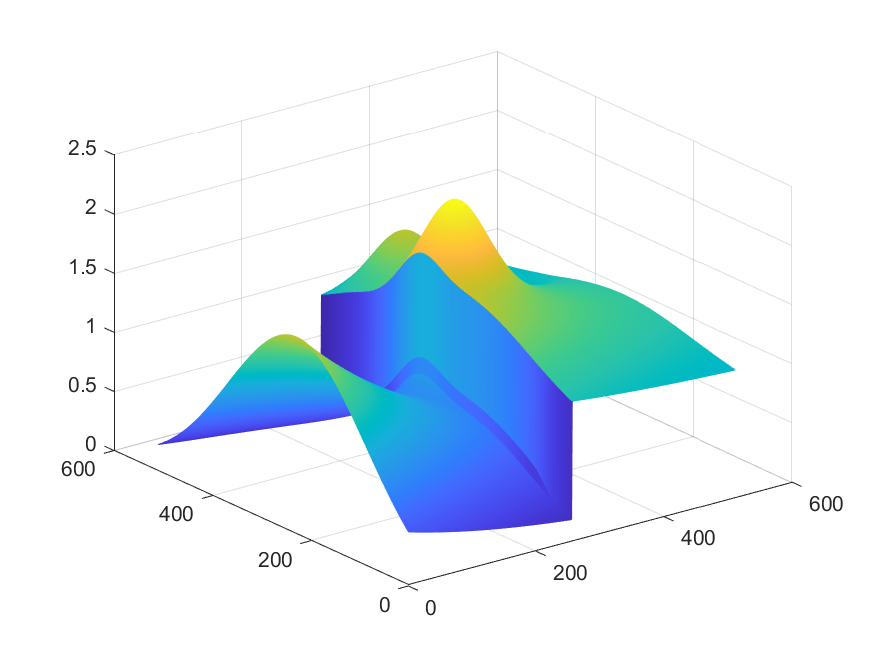} \\
    (a) & (b) & (c)\\
    \includegraphics[width=0.32\textwidth]{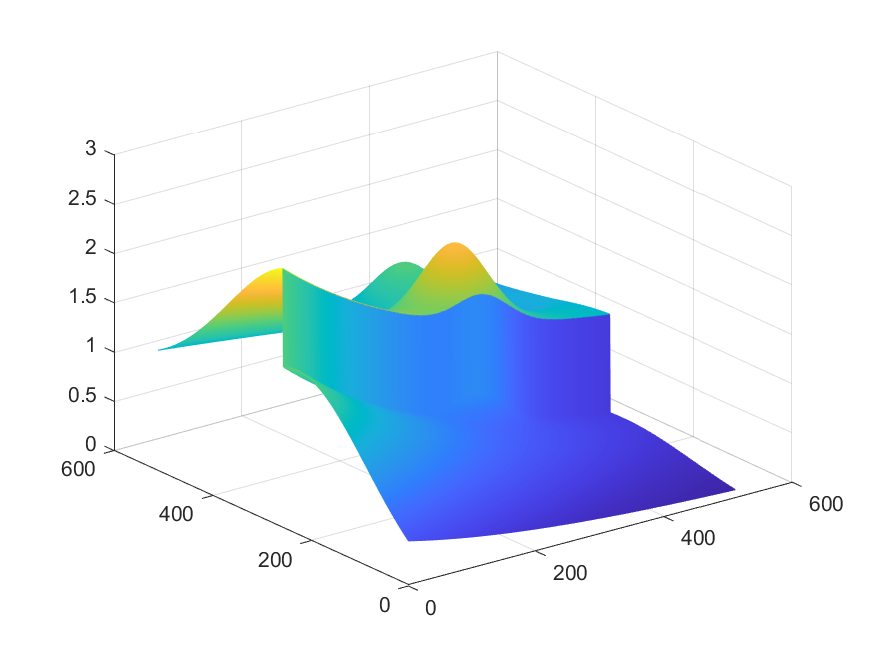} &
    \includegraphics[width=0.32\textwidth]{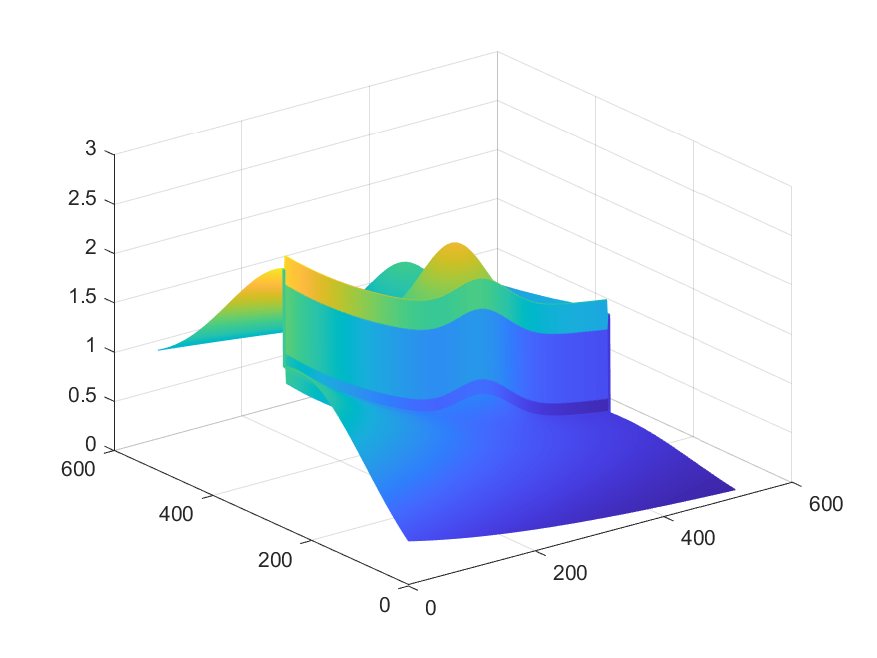} &
    \includegraphics[width=0.32\textwidth]{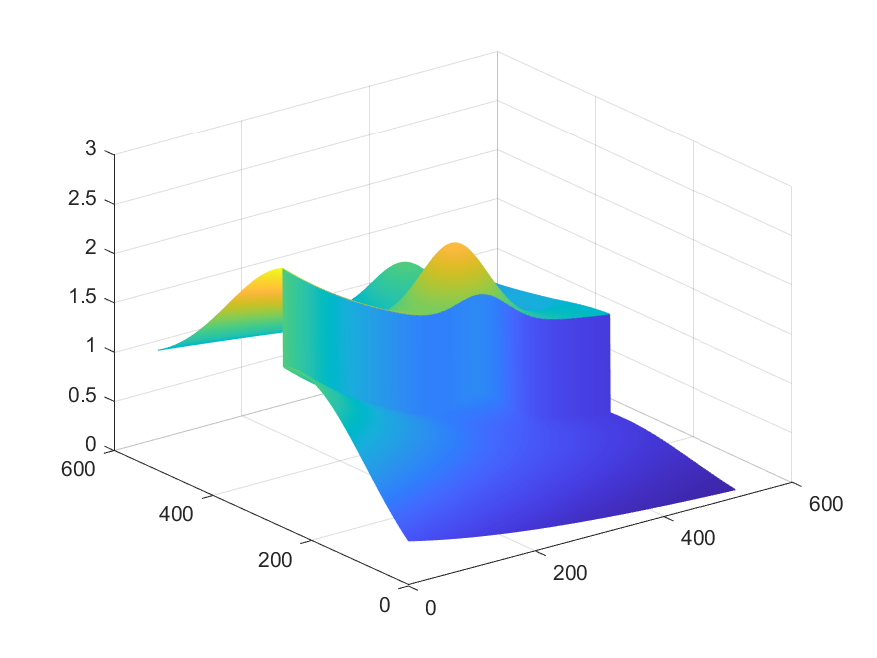} \\
    (d) & (e) & (f)
    \end{tabular}
    \caption{Three-dimensional visualization of the modified Franke surface \eqref{franke_disc} with a straight-line discontinuity. 
    The top row (panels (a)--(c)) corresponds to a horizontal jump along $y=0$: (a) original surface (horizontal discontinuity), (b) reconstruction with the linear predictor (smoothed jump), and (c) reconstruction with the progressive WENO method (edge preserved). 
    The bottom row (panels (d)--(f)) corresponds to a vertical jump along $x=0$: (d) original surface (vertical discontinuity), (e) linear reconstruction, and (f) WENO reconstruction.}
    \label{fig:franke_recon}
\end{figure}

The errors measured in the $\ell^2$ norm confirm this behavior: the linear scheme produced an error of $E_2^{\mathrm{LIN}} =$ {\color{black} 2.6357e-02} for both discontinuity cases, whereas the WENO reconstruction achieved $E_2^{\mathrm{WENO}} =$ {\color{black} 1.2795e-02} for the horizontal discontinuity and $E_2^{\mathrm{WENO}} =$ {\color{black} 3.8177e-06} for the vertical one. These results highlight the superior performance of WENO-type schemes in reconstructing surfaces with discontinuities.

In summary, the WENO predictor significantly reduces the $\ell^2$ error when the discontinuity is horizontal. For the vertical jump, WENO still outperforms the linear predictor but with a smaller improvement, likely due to stencil orientation effects.

{\color{black} The error values corresponding to the test functions are summarized in Table \ref{tab:errors_section62}.}
 
\begin{table}[H]
\centering
{\color{black}
\begin{tabular}{llll}
\hline
Function & Case & $E_2^{\rm LIN}$ & $E_2^{\rm WENO}$ \\ \hline
$g$ & -- & 4.2171e-01 & 1.8473e-05 \\
$h$ & -- & 2.9468e-02 & 7.8814e-03 \\
$f$ & Horizontal disc. & 2.6357e-02 & 1.2795e-02 \\
$f$ & Vertical disc. & 2.6357e-02 & 3.8177e-06 \\ \hline
\end{tabular}
\caption{Discrete $\ell^2$ errors for the three test functions using the linear and WENO reconstructions.}}
\label{tab:errors_section62}
\end{table}

\subsection{Digital image compression}

The objective in this {\color{black} subsection} is to evaluate the fidelity of the reconstruction in the multiresolution {\color{black} context} using both methods: linear and non-linear. For this goal, we measure  the number of non-zero coefficients (NNZ), as well as the global reconstruction errors of the $\ell^1$ and $\ell^2$ discrete norms, which we will define below. 

Unless otherwise stated, in almost all of our numerical tests, we take the following parameters: $L=4$, which specifies the number of multiresolution levels; $\varepsilon_L=30$, truncation constant; we only store the errors that are less than or equal to this value. When the resolution is reduced by one level, this value is divided by two. And the size  $N=512$, indicating the dimensions of the input matrices. Let $A\in \mathbb{R}^{N\times N}$ be the original image, we measure the error in the discrete norm $\ell^p$ with $p=1,2$ defined by:

\begin{align}
\mathbf{E}_{1} &= 
\frac{1}{N^2} \sum_{i,j=1}^{N} 
\left| \tilde{A}_{i,j} - A_{i,j} \right|, 
\label{eq:E1M}\\[6pt]
\mathbf{E}_{2} &=
\frac{1}{N}
\sqrt{ \sum_{i,j=1}^{N} 
\left( \tilde{A}_{i,j} - A_{i,j} \right)^2 },
\label{eq:E2M}
\end{align}
where $\tilde{A}$ is the image resulting from applying multiresolution to user-specified levels $L$ and $\varepsilon_L$. For colour images, rather than reporting the results for each channel (R, G, and B) independently, we compute the total NNZ as the sum across all channels and the average reconstruction errors for $\mathbf{E}_1$ and $\mathbf{E}_2$  are given by
\begin{equation} \label{E}
\mathbf{E}_i = \left( \mathbf{E}^{R}_i + \mathbf{E}^{G}_i + \mathbf{E}^{B}_i \right)/3, \quad i = 1, 2.
\end{equation}
It is important to note that, instead of representing a digital image as a single matrix $A \in \mathbb{R}^{N \times N}$, we incorporate an additional dimension to account for color. This leads to three separate matrices of size $N \times N$, corresponding to the three primary color channels that compose a digital photograph. Figure~\ref{originalC} displays the original RGB color images that will be used in our experiments.

\begin{figure}[h!]
     \centering
     \begin{tabular}{cc}
         \includegraphics[width=0.3\hsize]{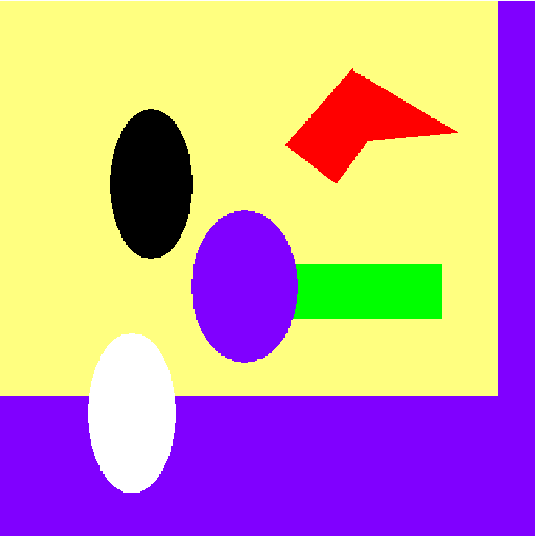} &
         \includegraphics[width=0.3\hsize]{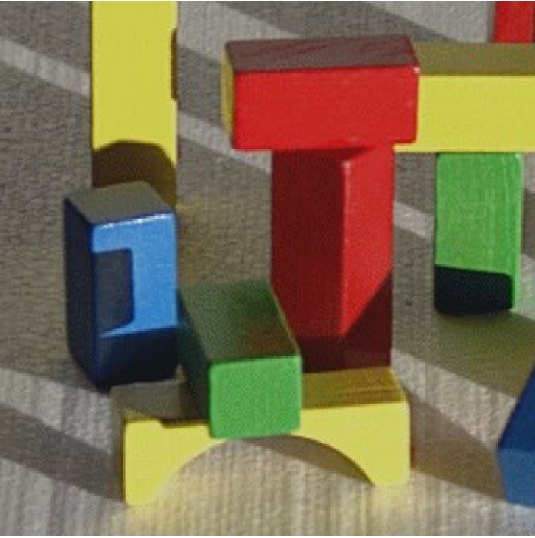}\\ 
         (a)  & (b)\\
         \includegraphics[width=0.3\hsize]{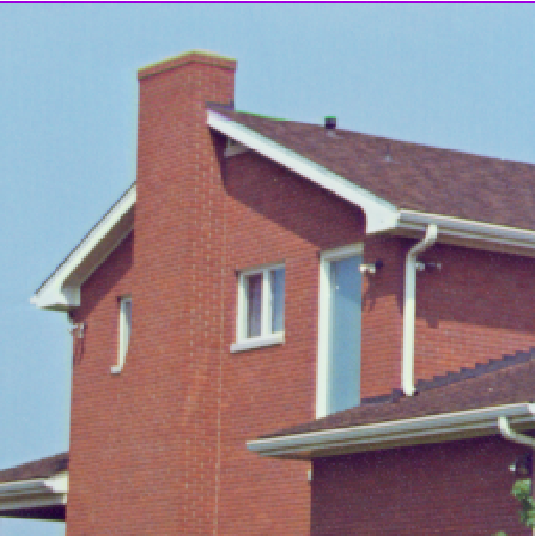} & \includegraphics[width=0.3\hsize]{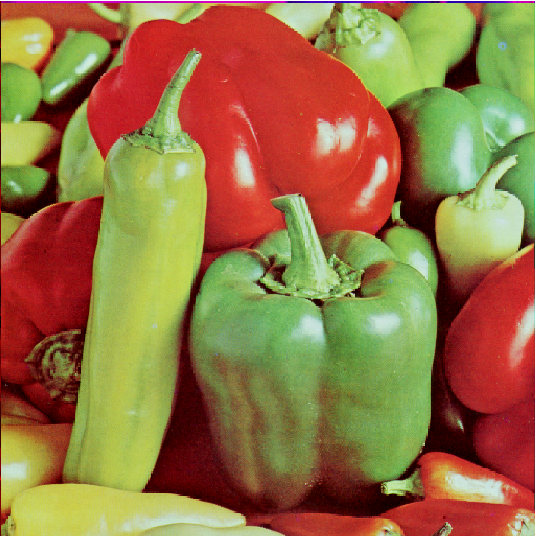}\\
         (c)  & (d)
     \end{tabular}
    \caption{Coloured images (RGB) in its original version \textbf{(a)} \textit{Geometric}; \textbf{(b)} \textit{Blocks}; \textbf{(c)} \textit{Red house}; \textbf{(d)} \textit{Peppers}.}
    \label{originalC}
\end{figure}


We begin our analysis with the first image shown in Figure~\ref{originalC}. This image, referred to as \textit{Geometric}, serves as the starting point for our study. This image has been previously analyzed in \cite{amatliantrijuanimag}, where its structural characteristics were highlighted as particularly challenging for reconstruction methods due to the presence of sharp edges and geometric patterns. In our study, the original image is represented as a tensor of size $256 \times 256 \times 3$, preserving the RGB color channels.

\begin{figure}[h!]
     \centering
     \begin{tabular}{cc}
         \includegraphics[width=0.3\hsize]{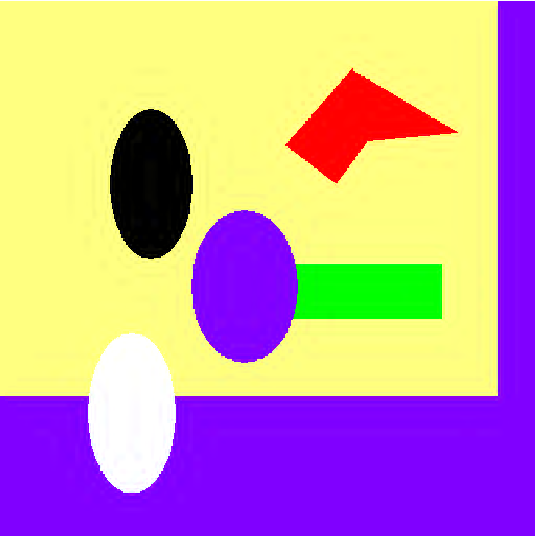} & \includegraphics[width=0.3\hsize]{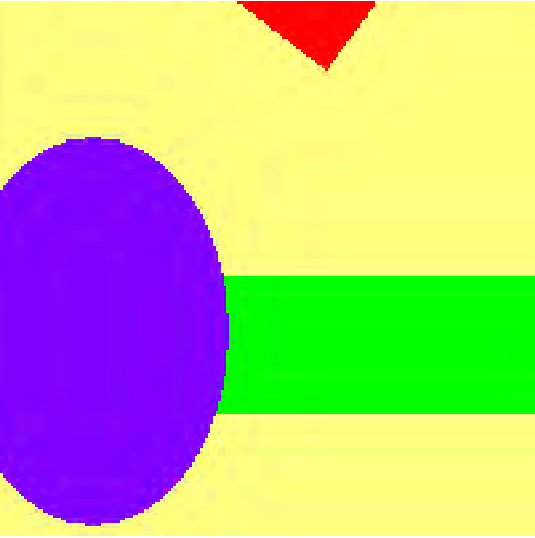} \\
         (a)  & (b) \\
         \includegraphics[width=0.3\hsize]{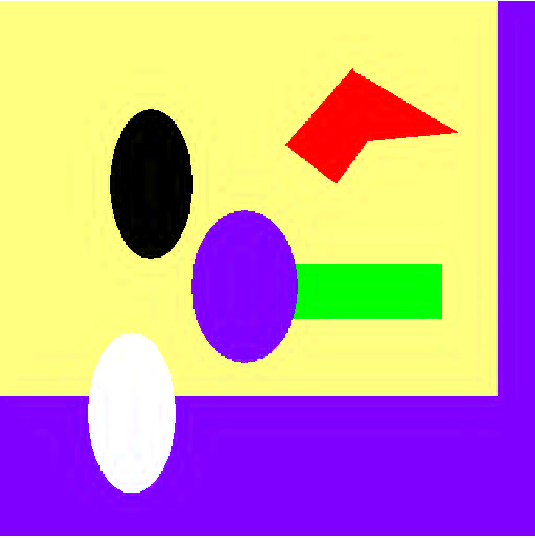} & \includegraphics[width=0.3\hsize]{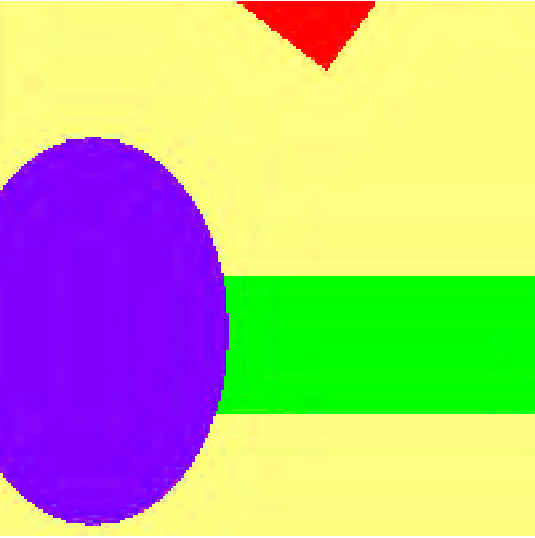} \\
         (c)  & (d) \\
     \end{tabular}
    \caption{\textit{Geometric} image reconstructed using the two methods: \textbf{(a)} linear reconstruction and \textbf{(c)} progressive bivariate WENO-2D reconstruction, with $\varepsilon_L=30$. Panels \textbf{(b)} and \textbf{(d)} show magnified (zoomed-in) views of the corresponding regions highlighted in (a) and (c), respectively, to illustrate local differences near edges and discontinuities.}
    \label{figura6}
\end{figure}

At first glance, Figure~\ref{figura6} does not reveal noticeable differences between the two reconstruction methods under comparison. Even when zooming in on specific regions of the image, these differences remain subtle and difficult to perceive visually. To gain a clearer understanding of the differences between the two methods, Table \ref{tabla6} summarises the information needed to draw conclusions. 


\begin{table}[h!]
\centering
\renewcommand{\arraystretch}{1.2}
\setlength{\tabcolsep}{6pt}
\begin{tabular}{lcccccccccccc}
\hline
 & \multicolumn{3}{c}{$\varepsilon_L = 5$} 
 & \multicolumn{3}{c}{$\varepsilon_L = 10$} 
 & \multicolumn{3}{c}{$\varepsilon_L = 20$} 
 & \multicolumn{3}{c}{$\varepsilon_L = 30$} \\
\hline
\textbf{Method} & $\mathbf{E_1}$ & $\mathbf{E_2}$ & \textbf{NNZ} 
 & $\mathbf{E_1}$ & $\mathbf{E_2}$ & \textbf{NNZ} 
 & $\mathbf{E_1}$ & $\mathbf{E_2}$ & \textbf{NNZ} 
 & $\mathbf{E_1}$ & $\mathbf{E_2}$ & \textbf{NNZ} \\[3pt]
\hline
LIN  & 0.2257 &0.6467 & 53562&0.5744 & 1.3755 & 39190 & 1.2040 & 2.7845  & 26171 &  1.6501 & 3.7677 & 21153   \\[3pt]
WENO & 0.2040 & 1.0532 & 24140& 0.5004&1.5535 & 21579 & 1.0298 & 2.4339 & 17530 & 1.5179 & 3.6463 & 13975\\[3pt]
\hline
\end{tabular}
\caption{
Results for the \textit{Geometric} image showing the values of the cell-average errors $\mathbf{E_1}$ and $\mathbf{E_2}$ (see Eqs.~\eqref{eq:E1M}--\eqref{E}), and the number of non-zero coefficients (NNZ) for different values of the parameter $\varepsilon_L$.}
\label{tabla6}
\end{table}

Table~\ref{tabla6} summarizes the results for the \textit{Geometric} image. In this case, the WENO reconstruction clearly outperforms the linear approach {\color{black} in terms of sparsity, achieving a reduction of approximately $33\%$ in the number of non-zero coefficients (NNZ) for $\varepsilon_L=30$, while maintaining comparable values of the errors $\mathbf{E}_1$ and $\mathbf{E}_2$}. This indicates that WENO attains a more efficient compression, reducing the number of active coefficients without compromising accuracy. The improvement is especially noticeable at higher compression levels, where the linear scheme exhibits a faster degradation of quality. Overall, these results highlight the strong adaptability of the WENO-based multiresolution framework when dealing with images that contain sharp edges and geometric patterns, confirming its robustness and effectiveness for high-fidelity image compression and reconstruction.

To further illustrate the numerical behavior of both reconstruction strategies, we proceed with the second image shown in Figure~\ref{originalC}, namely the colored \textit{Blocks} image.

\begin{figure}[h!]
     \centering
     \begin{tabular}{cc}
         \includegraphics[width=0.3\hsize]{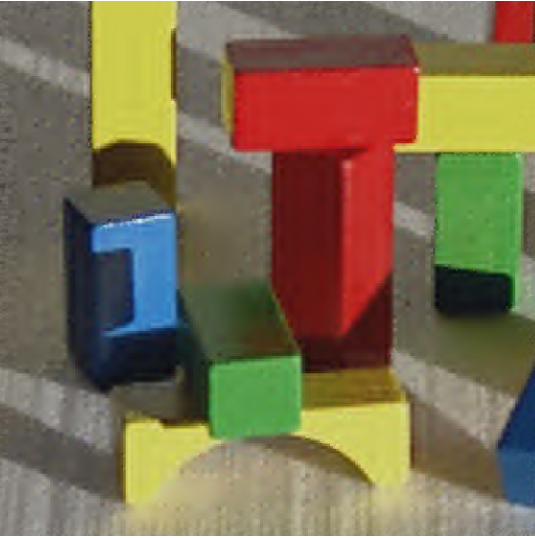} & \includegraphics[width=0.3\hsize]{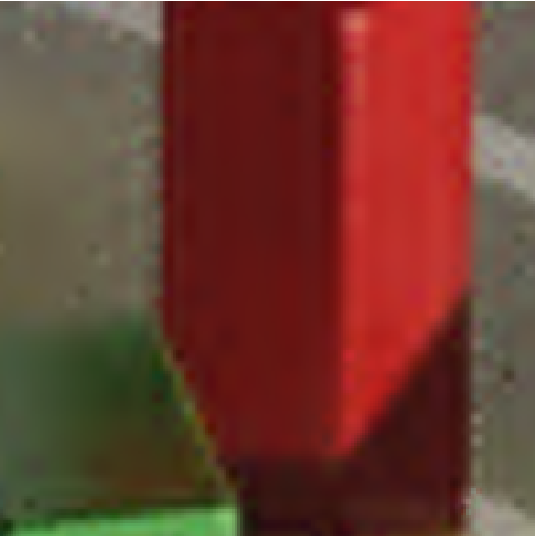} \\
         (a)  & (b) \\
         \includegraphics[width=0.3\hsize]{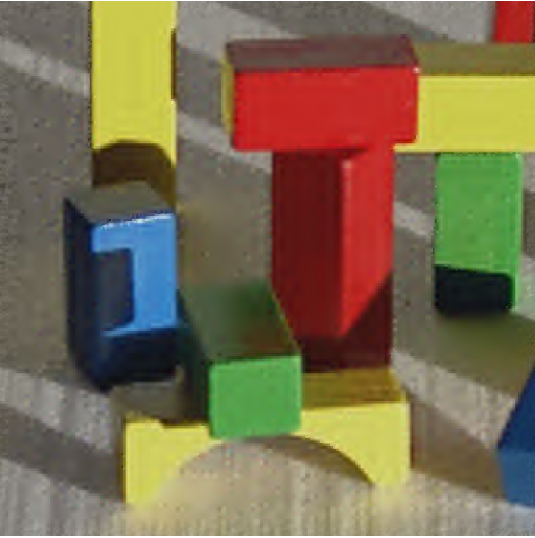} & \includegraphics[width=0.3\hsize]{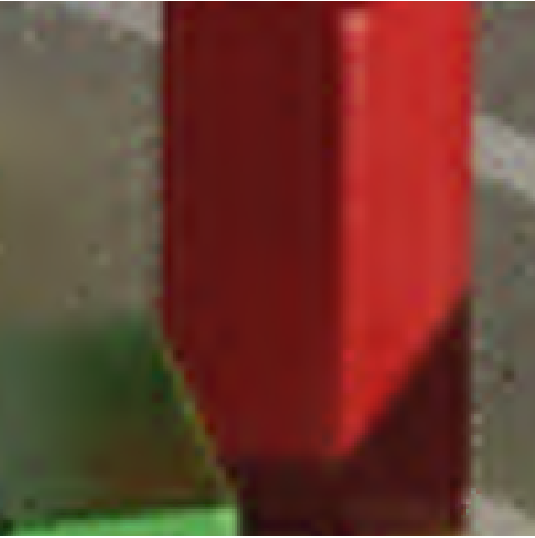} \\
         (c)  & (d) \\
     \end{tabular}
    \caption{\textit{Blocks} image reconstructed using the two methods: \textbf{(a)} linear reconstruction and \textbf{(c)} progressive bivariate WENO-2D reconstruction, with $\varepsilon_L=30$. Panels \textbf{(b)} and \textbf{(d)} show magnified (zoomed-in) views of the corresponding regions highlighted in (a) and (c), respectively, to illustrate local differences near edges and discontinuities.}
    \label{figura7}
\end{figure}

The visual results obtained by applying the linear and non-linear reconstruction methods are presented in Figure~\ref{figura7}. As in the previous case, the differences between the two approaches are not easily distinguishable through visual inspection alone, even when zooming into specific regions. Consequently, a quantitative comparison based on error metrics and sparsity (NNZ) is provided in Table \ref{tabla7} to enable a rigorous assessment of both methods.

\begin{table}[h!]
\centering
\renewcommand{\arraystretch}{1.2}
\setlength{\tabcolsep}{6pt}
\begin{tabular}{lcccccccccccc}
\hline
 & \multicolumn{3}{c}{$\varepsilon_L = 5$} 
 & \multicolumn{3}{c}{$\varepsilon_L = 10$} 
 & \multicolumn{3}{c}{$\varepsilon_L = 20$} 
 & \multicolumn{3}{c}{$\varepsilon_L = 30$} \\
\hline
\textbf{Method} & $\mathbf{E_1}$ & $\mathbf{E_2}$ & \textbf{NNZ} 
 & $\mathbf{E_1}$ & $\mathbf{E_2}$ & \textbf{NNZ} 
 & $\mathbf{E_1}$ & $\mathbf{E_2}$ & \textbf{NNZ} 
 & $\mathbf{E_1}$ & $\mathbf{E_2}$ & \textbf{NNZ} \\[3pt]
\hline
LIN  & 0.9507 & 1.2950 & 84741 & 1.9548 & 2.6007 & 42682 & 3.0937 &  4.1437 & 16241 & 3.7475  & 5.0802 &  9427  \\[3pt]
WENO & 1.2482 & 1.6772 & 91284 & 2.1258 & 2.8198  & 44954 & 3.1541 & 4.2252  & 16444 & 3.7751  & 5.1154 & 8978\\[3pt]
\hline
\end{tabular}
\caption{
Results for the \textit{Blocks} image showing the values of the cell-average errors $\mathbf{E_1}$ and $\mathbf{E_2}$ (see Eqs.~\eqref{eq:E1M}--\eqref{E}), and the number of non-zero coefficients (NNZ) for different values of the parameter $\varepsilon_L$.}
\label{tabla7}
\end{table}

Table~\ref{tabla7} presents the results for the \textit{Blocks} image, reporting the reconstruction errors ($\mathbf{E}_1$, $\mathbf{E}_2$) and the number of non-zero coefficients (NNZ) for different values of the truncation constant $\varepsilon_L$. As $\varepsilon_L$ increases, both methods exhibit the expected trend of rising errors and decreasing NNZ, reflecting the trade-off between accuracy and compression. For all thresholds, the linear predictor attains slightly smaller errors, whereas WENO produces nearly identical sparsity levels. The differences between the two schemes become minimal at higher compression levels, where WENO even achieves a marginal reduction in NNZ while maintaining comparable accuracy. Overall, both methods behave in a very similar manner for this test, indicating that WENO provides a stable and reliable alternative capable of preserving discontinuities effectively under strong compression.




Now, we consider the third image shown in Figure~\ref{originalC}, referred to as \textit{Red house}. In this case, the original image $(a)$ has dimensions $256 \times 256 \times 3$, where $N = 256$ corresponds to the spatial resolution and the third dimension accounts for the three color channels (RGB).

\begin{figure}[h!]
     \centering
     \begin{tabular}{cc}
         \includegraphics[width=0.3\hsize]{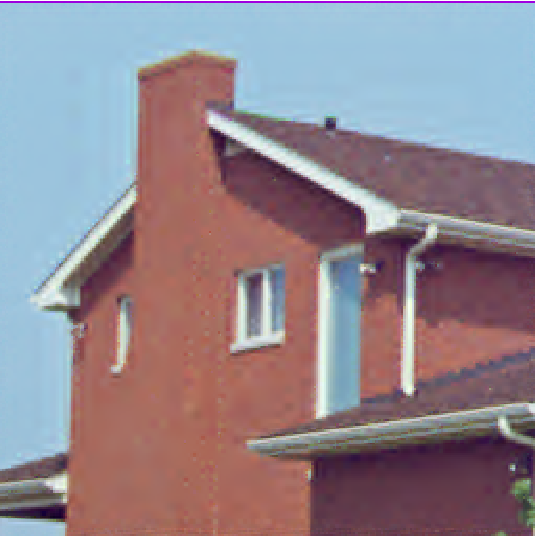} & \includegraphics[width=0.3\hsize]{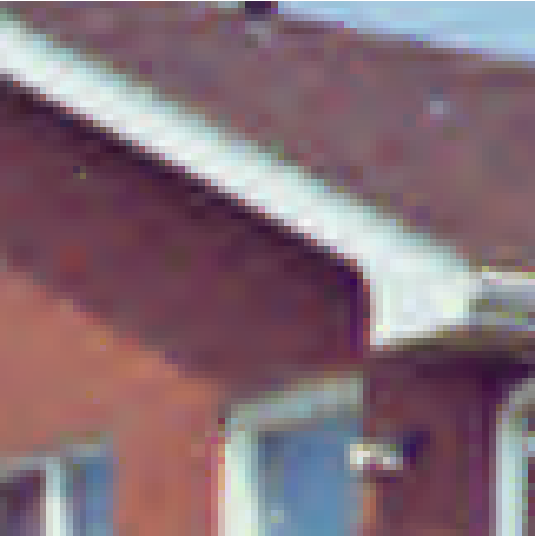} \\
         (a)  & (b) \\
         \includegraphics[width=0.3\hsize]{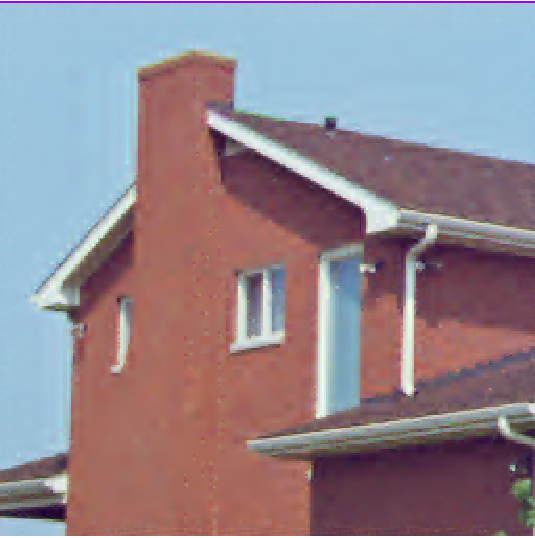} & \includegraphics[width=0.3\hsize]{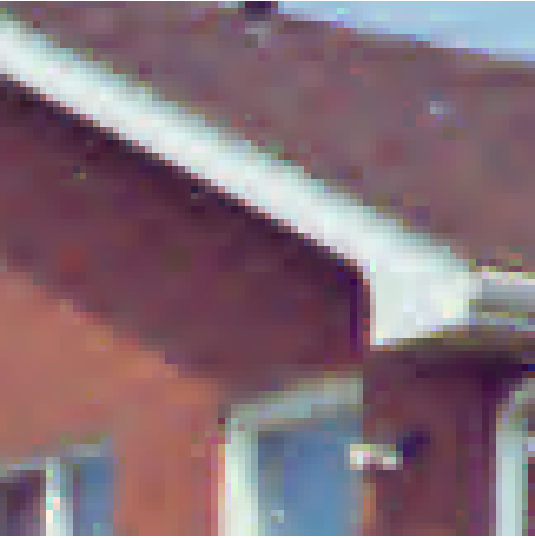} \\
         (c)  & (d) \\
     \end{tabular}
    \caption{\textit{Red house} image reconstructed using the two methods: \textbf{(a)} linear reconstruction and \textbf{(b)} progressive bivariate WENO-2D reconstruction, with $\varepsilon_L=30$. Panels \textbf{(c)} and \textbf{(d)} show magnified (zoomed-in) views of the corresponding regions highlighted in (a) and (b), respectively, to illustrate local differences near edges and discontinuities.}
    \label{figura4}
\end{figure}

We proceed analogously to the previous cases and present the corresponding table to highlight the differences between the two reconstruction methods that are not visually discernible in Figure \ref{figura4}.

\begin{table}[h!]
\centering
\renewcommand{\arraystretch}{1.2}
\setlength{\tabcolsep}{6pt}
\begin{tabular}{lcccccccccccc}
\hline
 & \multicolumn{3}{c}{$\varepsilon_L = 5$} 
 & \multicolumn{3}{c}{$\varepsilon_L = 10$} 
 & \multicolumn{3}{c}{$\varepsilon_L = 20$} 
 & \multicolumn{3}{c}{$\varepsilon_L = 30$} \\
\hline
\textbf{Method} & $\mathbf{E_1}$ & $\mathbf{E_2}$ & \textbf{NNZ} 
 & $\mathbf{E_1}$ & $\mathbf{E_2}$ & \textbf{NNZ} 
 & $\mathbf{E_1}$ & $\mathbf{E_2}$ & \textbf{NNZ} 
 & $\mathbf{E_1}$ & $\mathbf{E_2}$ & \textbf{NNZ} \\[3pt]
\hline
LIN  & 0.8846 & 1.2167 & 82358 & 1.7732 & 2.4516  & 45456  & 2.8884 & 4.0535 & 19231 & 3.5162 & 5.0635 &  11737   \\[3pt]
WENO & 1.2150 & 1.6652 & 91118 & 2.0398 & 2.8044  & 46110  & 3.0459 & 4.2792 & 19152 & 3.6141 & 5.2142 &  11864\\[3pt]
\hline
\end{tabular}
\caption{
Results for the \textit{Red house} image showing the values of the cell-average errors $\mathbf{E_1}$ and $\mathbf{E_2}$ (see Eqs.~\eqref{eq:E1M}--\eqref{E}), and the number of non-zero coefficients (NNZ) for different values of the parameter $\varepsilon_L$.}
\label{tabla4}
\end{table}

Table \ref{tabla4} compares the performance of LIN and WENO in terms of reconstruction errors ($\mathbf{E}_1$, $\mathbf{E}_2$) and sparsity (NNZ) for different values of the truncation constant $\varepsilon_L$. At lower thresholds, such as $\varepsilon_L = 5$ and $\varepsilon_L = 10$, the linear predictor produces smaller errors and slightly fewer non-zero coefficients, offering a more accurate and compact representation. In contrast, at higher thresholds, $\varepsilon_L = 20$ and $\varepsilon_L = 30$, the WENO reconstruction achieves lower errors and comparable or even smaller NNZ values. Overall, both methods behave similarly across the full range of compression levels, indicating that, in this particular case, the non-linear WENO adaptation performs on par with the linear scheme while maintaining robustness in regions with sharper transitions.


Finally, we conclude the analysis with the fourth image shown in Figure~\ref{originalC}, which corresponds to the well-known \textit{Peppers} test image characterized by smooth color variations and curved structures.


\begin{figure}[h!]
     \centering
     \begin{tabular}{ccc}
         \includegraphics[width=0.3\hsize]{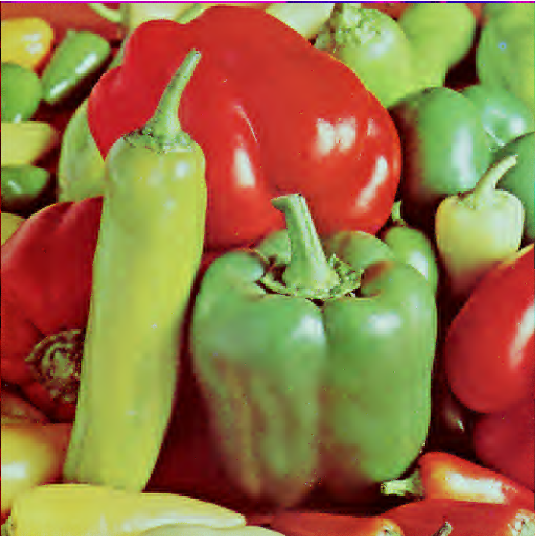} & \includegraphics[width=0.3\hsize]{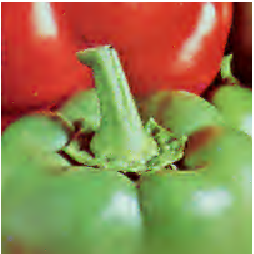} & \includegraphics[width=0.3\hsize]{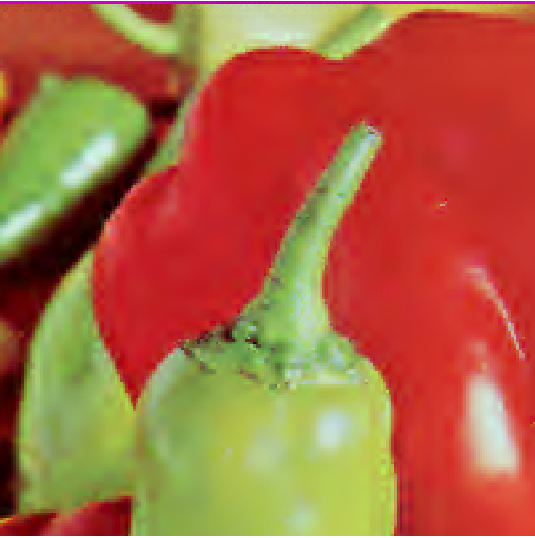}\\
          (a)  & (b)  & (c) \\
         \includegraphics[width=0.3\hsize]{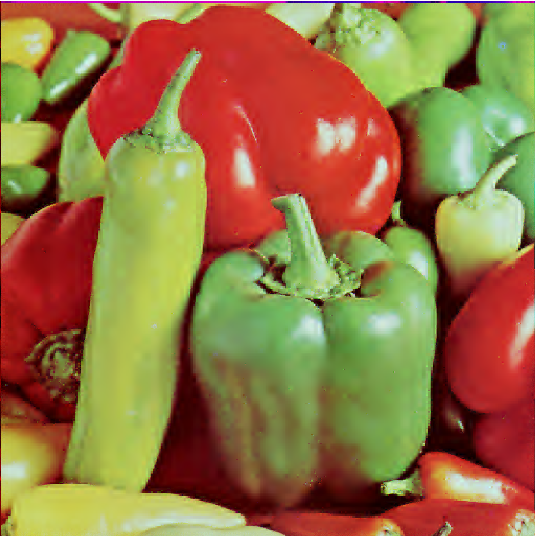} &  \includegraphics[width=0.3\hsize]{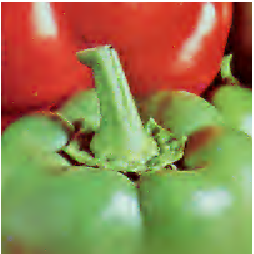} & \includegraphics[width=0.3\hsize]{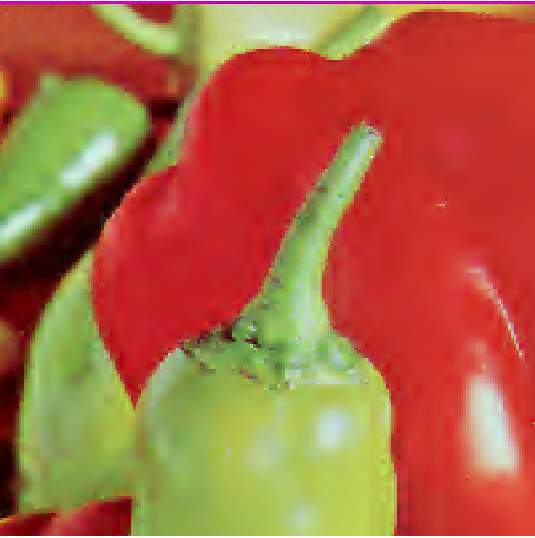}\\
         (d) & (e) & (f) \\
     \end{tabular}
    \caption{Reconstruction of the \textit{Peppers} image using linear and non-linear multiresolution methods. Panels \textbf{(a)--(c)} correspond to the linear reconstruction: (a) full image, and (b)--(c) zoomed-in regions highlighting edge smoothness. Panels \textbf{(d)--(f)} show the results obtained with the proposed non-linear progressive WENO method: (d) full image, and (e)--(f) corresponding zooms illustrating improved edge preservation and reduced artifacts.}

    \label{figura5}
\end{figure}

The results of applying the linear and non-linear reconstruction methods are shown in Figure~\ref{figura5}. However, as in the previous case, visual inspection alone does not reveal significant differences between the two approaches, even when zooming into specific regions. Therefore, a quantitative analysis based on error metrics and the number of non-zero coefficients (NNZ) is required to assess the performance of both methods properly. This information is provided in Table \ref{tabla5}.

\begin{table}[h!]
\centering
\renewcommand{\arraystretch}{1.2}
\setlength{\tabcolsep}{6pt}
\begin{tabular}{lcccccccccccc}
\hline
 & \multicolumn{3}{c}{$\varepsilon_L = 5$} 
 & \multicolumn{3}{c}{$\varepsilon_L = 10$} 
 & \multicolumn{3}{c}{$\varepsilon_L = 20$} 
 & \multicolumn{3}{c}{$\varepsilon_L = 30$} \\
\hline
\textbf{Method} & $\mathbf{E_1}$ & $\mathbf{E_2}$ & \textbf{NNZ} 
 & $\mathbf{E_1}$ & $\mathbf{E_2}$ & \textbf{NNZ} 
 & $\mathbf{E_1}$ & $\mathbf{E_2}$ & \textbf{NNZ} 
 & $\mathbf{E_1}$ & $\mathbf{E_2}$ & \textbf{NNZ} \\[3pt]
\hline
LIN  & 2.0410 & 2.7260 & 248514 & 3.6066 & 4.7234 & 90166  & 4.8058 & 6.4947 & 30356 & 5.4736 & 7.5959 &   17409  \\[3pt]
WENO & 2.1114 & 2.8487 & 247205 & 3.6544 & 4.8285 & 89663  & 4.8307 & 6.6077 & 31289 & 5.5066 & 7.7393 & 18340 \\[3pt]
\hline
\end{tabular}
\caption{
Results for the \textit{Peppers} image showing the values of the cell-average errors $\mathbf{E_1}$ and $\mathbf{E_2}$ (see Eqs.~\eqref{eq:E1M}--\eqref{E}), and the number of non-zero coefficients (NNZ) for different values of the parameter $\varepsilon_L$.}
\label{tabla5}
\end{table}

Table~\ref{tabla5} shows results for the \textit{Peppers} image. In this case, the trend is opposite to that observed for the \textit{Red house} image. For lower thresholds, such as $\varepsilon_L = 5$ and $\varepsilon_L = 10$, the WENO reconstruction produces slightly smaller errors than the linear predictor, achieving a marginal improvement in accuracy with a similar number of non-zero coefficients. However, for higher thresholds ($\varepsilon_L = 20$ and $\varepsilon_L = 30$), the linear method attains marginally lower errors, while both approaches exhibit nearly identical sparsity patterns. Overall, the two reconstructions behave in a very similar way across the full range of compression levels, indicating that their performance is practically equivalent for this image, with WENO maintaining good stability under strong compression.

\section{Conclusions and future work}

In this article, we have presented a non-linear progressive-order interpolation algorithm that improves upon the WENO method in the context of cell averages. This type of data discretization is particularly suitable for applications such as digital image processing. The main difference with respect to standard WENO methods lies in the evaluation step, which involves computing an average on a finer discretization scale.

We have applied the proposed methods to digital image compression using Harten’s multiresolution framework \cite{Har96}. This new approximation approach opens the possibility of modifying the type of discretization employed; for instance, one could use other kinds of averages, such as hat averages or any convolution thereof.

Such techniques can also be applied to the numerical solution of partial differential equations. Moreover, other interpolation strategies, such as radial basis function (RBF) methods (see \cite{cavo16}), can be implemented using the same algorithmic framework introduced here.

In summary, this generalization opens a wide range of possibilities, both from a theoretical and a practical point of view.

%
%
%

\end{document}